\documentclass[11pt]{article}
\usepackage[english]{babel}
\usepackage{ctex} 
\usepackage{amsfonts}
\usepackage{mathrsfs}
\usepackage{bbm}
\usepackage{amsfonts}
\usepackage{multirow}
\usepackage{amssymb,amsmath,graphicx}
\usepackage{float}
\usepackage[colorlinks=true]{hyperref}
\usepackage{tikz}
\usepackage{setspace}
\usepackage{fancyhdr} 
\usepackage{textcase}       
\usepackage{mathrsfs}
\usepackage{amsmath,amssymb}
\usepackage{amsthm}
\usepackage{mathtools} 
\usepackage{amssymb} 
\usepackage{bbm} 
\usepackage{enumitem}
\usepackage{bm}
\usepackage{hyperref}
\usepackage{cite} 
\usepackage{url}
\usepackage{algorithm2e} 
\hypersetup{urlcolor=blue, citecolor=red}
\usepackage{pifont}
\usepackage[perpage,symbol*,bottom]{footmisc}
\DefineFNsymbols{circled}{{\ding{172}}{\ding{173}}{\ding{174}}
{\ding{175}}{\ding{176}}{\ding{177}}{\ding{178}}{\ding{179}}{\ding{180}}{\ding{181}}}
\setfnsymbol{circled}
\usepackage{scrextend}
\hypersetup{hidelinks}
\usepackage{longtable}
\usepackage{tikz}
\usepackage{array} 
\usepackage{caption} 
\usepackage{booktabs}
\usetikzlibrary{trees}

\allowdisplaybreaks[1]

\allowdisplaybreaks
\usepackage{cleveref} 
\crefname{theorem}{Theorem}{Theorems}
\crefname{table}{Table}{Tables} 
\crefname{lemma}{Lemma}{Lemmas}
\crefname{definition}{Definition}{Definitions}
\begin{document}
\title{\textbf{A fractional attraction-repulsion chemotaxis system with generalized logistic source and nonlinear productions}
}\author{Liyan Song$^{1}$, Qingchun Li$^{1}$, Chengyuan Qu$^{2,*}$
\footnotemark[0]\\
\footnotesize $^{1}$School of Mathematical Sciences, Dalian University of Technology,
Dalian 116024, P. R. China \\
\footnotesize $^{2}$School of Mathematical Sciences, Dalian Minzu University,
Dalian 116600, P. R. China}
\footnotetext[0]{*Corresponding author. {\it E-mail}:  mathqcy@163.com}
\date{}

\maketitle
{\bf Abstract }
This paper studies a fractional attraction-repulsion system with generalized logistic source and nonlinear productions:
\begin{equation*}
\left\{
\begin{aligned}
&u_t = -(-\Delta)^\alpha u - \chi_1 \nabla \cdot (u \nabla v) + \chi_2 \nabla \cdot (u \nabla w) + au - bu^\gamma, &x \in \mathbb{R}^N, \, t > 0, \\
&0 = \Delta v - \lambda_1 v + \mu_1 u^k, &x \in \mathbb{R}^N, \, t > 0, \\
&0 = \Delta w - \lambda_2 w + \mu_2 u^k, &x \in \mathbb{R}^N, \, t > 0.
\end{aligned}
\right.
\end{equation*}\\
We first establish the global boundedness of classical solutions with nonnegative bounded and uniformly continuous initial data in two different cases: $\gamma \geq k + 1$ and $\gamma < k + 1$, respectively. Next, we show the asymptotic behavior of the global solutions for both cases $\gamma = k + 1$ and $\gamma \neq k + 1$. Finally, we obtain the spreading speed of solutions. In particular, when $\gamma = k + 1$, the upper bound of the spreading speed increases monotonically with $k$. If the condition of balanced attraction-repulsion intensities is further specified, the spreading speed will be equal to $\frac{a}{N + 2\alpha}$.


\newtheoremstyle{mytheoremstyle}
{3pt}
{3pt}
{\upshape}
{0pt}
{\bfseries}
{}
{.5em}
{}

\renewcommand{\proofname}{\textup{Proof}}
\theoremstyle{plain}
\newtheorem{theorem}{Theorem}[section] 
\newtheorem{lemma}{Lemma}[section] 
\newtheorem{definition}[lemma]{Definition} 
\newtheorem{proposition}{Proposition}[section]
\newtheorem{corollary}{Corollary}[section]
\newtheorem{remark}{Remark}
\renewcommand{\theequation}{\thesection.\arabic{equation}}
\catcode`@=11 \@addtoreset{equation}{section} \catcode`@=12
\maketitle{}

\section{Introduction}
In biological systems, the collective behavior of cells is regulated by distinct mechanisms including attraction-repulsion chemotaxis, fractional-order dynamics, generalized logistic source and nonlinear production. Among these, the dual role of attraction-repulsion chemotaxis serves as a key mechanism for preventing excessive aggregation. As demonstrated by Luca et al. (2003) \cite{LUCA} in the context of Alzheimer's disease, microglial cells are influenced by the chemotactic attraction toward $\beta$-amyloid and repulsive signals from TNF-$\alpha$. When the attraction strength falls below the repulsion threshold, amyloid plaque formation is inhibited. Building on a chemotaxis model incorporating chemoattractant-repellent dynamics, Quinlan \& Straughan (2005) \cite{QB} employed nonlinear stability analysis to identify threshold conditions for microglial aggregation such as the absence of amyloid plaques below a critical cell density. This model was validated through comparative analysis of microglial density in Alzheimer's versus healthy brains, confirming its robustness and supporting its relevance in explaining senile plaque formation in Alzheimer's disease.

Compared with traditional dynamical models, fractional-order dynamics exhibits significant advantages in addressing memory-dependent and non-instantaneous response problems. For instance, in the movement of microbial populations, it can accurately characterize the delayed perception of chemical signals by cells, providing crucial support for revealing the intrinsic laws of this process. Garfinkel et al. (2004) \cite{AYD} showed that the movement of mesenchymal cells in response to certain chemical attractants deviates from classical chemotaxis models. These models fail to explain the experimental observation that vascular mesenchymal cells can aggregate into one-dimensional structures in two-dimensional settings. Escudero (2006)\cite{EC} further highlighted that many organisms (e.g., microzooplankton, soil amoebae) adopt Lévy flights (super diffusive behavior) as a movement strategy, which cannot be described by classical random-walk models.

To quantify the ``aggregation-inhibition'' dynamics and derive the threshold conditions for biological aggregation, generalized logistic sources and nonlinear production functions provide reliable tool support. Based on the core logic of these tools, in plant root growth, Richards extended the growth equation originally developed by Von Bertalanffy to better fit empirical plant data \cite{RF}, accommodating S-shaped growth curves across diverse plant morphologies. Moreover, the balance between chemotactic attraction toward water and inhibitory effects due to rhizosphere competition enables precise prediction via thresholds derived from generalized logistic sources of whether root systems will become overly dense and lead to nutrient deficiency. The seminal work by Budrene \& Berg (1991) \cite{EH} first demonstrated that Escherichia coli can self-organize into distinct spatial patterns (e.g., spots and stripes) during proliferation, while systematically characterizing their morphological features and dynamic responses to substrate concentration gradients. Models incorporating nonlinear signal generation are commonly employed to elucidate the formation mechanisms of biological patterns, which further demonstrates that nonlinear production is indispensable for constructing biologically plausible models.

The present paper focuses on the analysis of a fractional attraction-repulsion chemotaxis system involving nonlinear signal production mechanisms and generalized logistic source
\begin{equation}\label{1.1} 
\left\{
\begin{aligned}
&u_t = -(-\Delta)^\alpha u - \chi_1 \nabla \cdot (u \nabla v) + \chi_2 \nabla \cdot (u \nabla w) + au - bu^\gamma, \quad &x \in \mathbb{R}^N, \, t > 0, \\
&0 = \Delta v - \lambda_1 v + \mu_1 u^k, \quad &x \in \mathbb{R}^N, \, t > 0, \\
&0 = \Delta w - \lambda_2 w + \mu_2 u^k, \quad &x \in \mathbb{R}^N, \, t > 0,
\end{aligned}
\right.
\end{equation}
where $N \geq 1$, $\gamma > 1$, $\alpha \in ( \frac{1}{2}, 1 )$, $k\geq1$, $\chi_i \geq 0\ (i = 1,2)$ and $a, b, k, \mu_1, \lambda_2, \mu_2$ are positive constants. Here $u(x,t)$ denotes the mobile species concentration, while $v(x,t)$ and $w(x,t)$ represent the chemical concentrations of the attractant and repellent, respectively. $\chi_1$ and $\chi_2$ measure the chemotactic sensitivity of the mobile species. $\lambda_1$ and $\lambda_2$ describe the degradation rates of chemical substances and $\mu_1$, $\mu_2$ represent the production rates of chemical substances by the mobile species.

Since the groundbreaking work of Keller and Segel in the 1970s \cite{KS,KSS}, the classical chemotaxis system has been widely studied and extended \cite{NAYM,Y,YH,M,GJ,D,K}. In numerous biological processes, general mechanisms in cells include not only chemoattractants but also chemorepellents, which can be expressed as the following attraction-repulsion chemotaxis system
\begin{equation}\label{1.2}
\left\{
\begin{aligned}
&u_t = \Delta u - \chi_1 \nabla \cdot (u \nabla v) + \chi_2 \nabla \cdot (u \nabla w) + f(u), & x \in \Omega, \, t > 0, \\
&0 = \Delta v - \lambda_1 v + \mu_1 u, & x \in \Omega, \, t > 0, \\
&0 = \Delta w - \lambda_2 w + \mu_2 u, & x \in \Omega, \, t >0,
\end{aligned}
\right.
\end{equation}\\
where $\Omega \subset \mathbb{R}^N (N \geq 1)$ is a smooth bounded domain. For the case of $f(u) \leq au - b u^2$, Zhang \& Li (2016)\cite{ZL} proved the existence of a unique global solution under appropriate conditions. When $f(u) \leq a-bu^{\gamma}$, Xie \& Xiang (2016)\cite{LX} showed that \eqref{1.2} has a globally bounded classical solution for any nonnegative initial function $u_0 \in W^{1,\infty}(\Omega)$ if appropriate assumptions hold. Moreover, many researchers have adopted nonlinear signal production $\mu u^k$ over the linear form $\mu u$ to better model chemotactic cell migration in response to chemical gradients \cite{EM,EK,EKK}. Hong et al. (2020)\cite{LMS} showed the existence of a globally bounded classical solution by analyzing different cases depending on the relationship between $\gamma$ and $k$. In particular, Zhou et al. (2022)\cite{XZJ} studied the asymptotic behavior of globally bounded solutions with $f(u) =au-bu^{\gamma}$ when $\gamma=k+1 \geq \frac{2}{N}$.
For $\Omega = \mathbb{R}^N$, Salako \& Shen (2019) \cite{BW} investigated the asymptotic stability and propagation dynamics of solutions to \eqref{1.2} when $f(u) = au - bu^2$. However, only limited results have been obtained for the Cauchy problem with both generalized logistic source and nonlinear production term in \eqref{1.2} in the whole space. As far as we know, Hassan et al. (2025)\cite{HZWY} studied the chemotaxis-consumption system with logistic term $au - bu^\gamma$ and established the global boundedness of classical solution.

With the in-depth investigation of chemotaxis, the introduction of fractional-order effects into chemotaxis models has attracted significant research attention. In this context, many researchers have studied the following fractional attraction-repulsion chemotaxis system
\begin{equation}\label{1.3}
\left\{
\begin{aligned}
&u_t = -(-\Delta)^\alpha u - \chi_1 \nabla \cdot (u \nabla v) + \chi_2 \nabla \cdot (u \nabla w) + u(a - bu), & x \in \mathbb{R}^N, \, t > 0, \\
&0 = \Delta v- \lambda_1 v + \mu_1 u, & x \in \mathbb{R}^N, \, t > 0, \\
&0 = \Delta w- \lambda_2 w + \mu_2 u, & x \in \mathbb{R}^N, \, t > 0.
\end{aligned}
\right.
\end{equation}
Zhang et al. (2019)\cite{ZZL} studied the global boundedness of classical solutions and the asymptotic stability of positive constant equilibria when $w\equiv 0$. Jiang et al. (2022)\cite{JZZ} obtained the global well-posedness of the solution to \eqref{1.3} when $\chi_2\mu_2 < \chi_1\mu_1$ and $f(u)=0$. In particular, Jiang et al. (2023)\cite{JLL} established the global existence, asymptotic behavior and spreading properties of global classical solutions for \eqref{1.3}.

Inspired by \cite{BW,JLL,XZJ,LMS,SX,ZZL}, we investigate a fractional attraction-repulsion chemotaxis system with generalized logistic source and nonlinear production of \eqref{1.1} in the whole space. The core objective of this study is to conduct an in-depth analysis of three key mechanisms in the chemotaxis process (attraction-repulsion, fractional-order dynamics, generalized logistic source and nonlinear production) and show their regulatory effects on the behavior of solutions to \eqref{1.1}. Now, we present the main results as follows.

Throughout this paper, let
\begin{align*}
C_{unif}^b(\mathbb{R}^N) = \{ u \in C(\mathbb{R}^N) \mid u(x) \text{ is uniformly continuous in } x \in \mathbb{R}^N \text{and} \sup_{x \in \mathbb{R}^N} |u(x)| < \infty \}
\end{align*}
provided with the norm $\| u \|_{L^\infty} = \sup_{x \in \mathbb{R}^N} | u(x) |$.

First, we give the local existence and uniqueness of classical solutions to \eqref{1.1} with nonnegative initial data \( u_0 \in C_{unif}^b \left( \mathbb{R}^N \right) \) by using the contraction mapping theorem \cite{JLL}. We omit the proof here.
\begin{proposition}\label{pro1.1}
Let $\gamma > 1$, $k\geq1$, $\alpha \in ( \frac{1}{2}, 1)$ and $0 \leq u_0 \in C_{unif}^b \left( \mathbb{R}^N \right)$. Then there exist a maximal existence time $T_{\max}$ and a unique nonnegative classical solution $(u, v, w)$ on $[0, T_{\max})$ such that $\lim_{t \to 0^+} u(\cdot, t) = u_0$ and
\[
u \in C \left( [0, T_{\max}), C_{unif}^b \left( \mathbb{R}^N \right) \right) \cap C^1 \left( (0, T_{\max}), C_{unif}^b \left( \mathbb{R}^N \right) \right).
\]
Moreover, if $T_{\max} < \infty$, then
\[
\limsup_{t \to T_{\max}} \| u(\cdot, t) \|_{L^\infty} = \infty.
\]
\end{proposition}

\begin{theorem}[Global Boundedness]\label{thm:mytheorem1.2}
Let $\gamma > 1$, $\alpha \in ( \frac{1}{2}, 1 )$, $k\geq1$, $\chi_1, \chi_2, a, b > 0$, $u_0 \in C_{unif}^b \left( \mathbb{R}^N \right)$ and $\inf\limits_{x \in \mathbb{R}^N} u_0(x) > 0$. Then \eqref{1.1} possesses a unique nonnegative global classical solution $(u, v, w)$ if one of the following assumptions holds:
\begin{enumerate}
\item[\rm(a)] $\gamma \geq k + 1$, $b + \chi_2 \mu_2 - \chi_1 \mu_1 - M > 0$;
\item[\rm(b)] $\gamma < k + 1$, $\chi_2 \lambda_2 \mu_2 > \chi_1 \lambda_1 \mu_1$, $\lambda_1 > \lambda_2$;
\item[\rm(c)] $\gamma < k + 1$, $\| u_0 \|_{L^\infty} \leq ( \frac{b - a}{M + \chi_1 \mu_1} )^{\frac{1}{k}}$, $b > a + M + \chi_1 \mu_1$;
\item[\rm(d)] $\gamma \neq k + 1$, $\chi_2 \mu_2 = \chi_1 \mu_1$, $\lambda_1 = \lambda_2$,
\end{enumerate}
where
\begin{align*}
M := \min \biggl\{ &\frac{1}{\lambda_1} \Bigl( (\chi_2 \mu_2 \lambda_2 - \chi_1 \mu_1 \lambda_1)_+ + \chi_2 \mu_2 (\lambda_1 - \lambda_2)_+ \Bigr), \\
&\frac{1}{\lambda_2} \Bigl( (\chi_2 \mu_2 \lambda_2 - \chi_1 \mu_1 \lambda_1)_+ + \chi_1 \mu_1 (\lambda_1 - \lambda_2)_+ \Bigr) \biggr\}.
\end{align*}
Furthermore, it holds that $\| u \|_{L^\infty} \leq C_0$, where
\begin{equation*}
C_0 =
\begin{cases}
\| u_0 \|_{L^\infty}, & \text{\rm if } \chi_i = a = b = 0, \\
\max \left\{ 1, \| u_0 \|_{L^\infty}, \left( \frac{a}{b + \chi_2 \lambda_2 - \chi_1 \mu_1 - M} \right)^{\frac{1}{k}} \right\}, &  if ~ \mathrm{(a)} ~holds, \\
\max \left\{ 1, \| u_0 \|_{L^\infty}, \left( \frac{a}{b} \right)^{\frac{1}{\gamma - 1}} \right\}, & if ~ \mathrm{(b)} ~holds, \\
\max \left\{ 1, \| u_0 \|_{L^\infty}, C_* \right\}, & if ~ \mathrm{(c)} ~holds, \\
\max \left\{1, \| u_0 \|_{L^\infty}, \left( \frac{a}{b} \right)^{\frac{1}{\gamma - 1}} \right\}, & if ~ \mathrm{(d)} ~holds
\end{cases}
\end{equation*}
and \( C_* \leq \left( \frac{b - a}{M + \chi_1 \mu_1} \right)^{\frac{1}{k}} \).
\end{theorem}
\begin{table}[H]
\centering
\small 
\caption{The global boundedness of classical solutions}
\label{tab:your_label} 
\begin{tabular}{
>{\centering\arraybackslash}m{4.8cm} 
>{\centering\arraybackslash}m{3.8cm}
>{\centering\arraybackslash}m{5cm}
}
\toprule
& $(\mathrm{a})\ \gamma \geq k + 1$ & $(\mathrm{c})\ \gamma < k + 1$ \\
\midrule
$\begin{aligned}[c]
\chi_2\lambda_2\mu_2 &\geq \chi_1\lambda_1\mu_1, \\
\lambda_1 &\geq \lambda_2
\end{aligned}$
& $b > 0$
& $b > a + \chi_2\mu_2$ \\
\addlinespace[0.6em] 
$\begin{aligned}[c]
\chi_2\lambda_2\mu_2 &\geq \chi_1\lambda_1\mu_1, \\
\lambda_1 &\leq \lambda_2
\end{aligned}$
& $\begin{aligned}[t]
b &> \chi_1\mu_1 \biggl(1 - \frac{\lambda_1}{\lambda_2}\biggr)
\end{aligned}$
& $\begin{aligned}[t]
b &> a + \chi_1\mu_1 \biggl(1 - \frac{\lambda_1}{\lambda_2}\biggr) + \chi_2\mu_2
\end{aligned}$ \\
\addlinespace[0.6em]
$\begin{aligned}[c]
\chi_2\lambda_2\mu_2 &\leq \chi_1\lambda_1\mu_1, \\
\lambda_1 &\geq \lambda_2
\end{aligned}$
& $b > \chi_1\mu_1 - \chi_2\mu_2\frac{\lambda_2}{\lambda_1}$
& $b > a + \chi_1\mu_1 + \chi_2\mu_2\biggl(1 - \frac{\lambda_2}{\lambda_1}\biggr)$ \\
\addlinespace[0.6em]
$\begin{aligned}[c]
\chi_2\lambda_2\mu_2 &\leq \chi_1\lambda_1\mu_1, \\
\lambda_1 &\leq \lambda_2
\end{aligned}$
& $b > \chi_1\mu_1 - \chi_2\mu_2$
& $b > a + \chi_1\mu_1$ \\
\bottomrule
\end{tabular}
\end{table}

\medskip
\noindent\textbf{Remark 1.1} As shown in Columns 2 and 3 of \cref{tab:your_label}, comparing the cases of $\gamma \geq k+1$ and $\gamma < k+1$ for a fixed $\gamma$, a larger value of $b$ must be chosen to avoid the blow-up of the solution. This reflects that an increase in the parameter $k$ enhances cell secretion of chemical signaling molecules which attracts more cells and promotes population growth. The term $-bu^{\gamma}$ restricts cellular proliferation, requiring increased damping parameter $b$ to prevent the blow-up of solution when $\gamma$ cannot counterbalance the growth promotion by the nonlinear exponent $k$.\\
\noindent\textbf{Remark 1.2} Case (d) confirms the global boundedness of solutions in the critical case which consequently guarantees the asymptotic stability of constant equilibria.

\begin{theorem}[Asymptotic Behavior]\label{thm:mytheorem1.3}
Let $\gamma > 1$, $\alpha \in ( \frac{1}{2}, 1)$, $k\geq1$, $\chi_1, \chi_2, a, b>0$, $u_0 \in C_{unif}^b \left( \mathbb{R}^N \right)$ and $\inf\limits_{x \in \mathbb{R}^N} u_0(x) > 0$. Then the unique global classical solution $(u, v, w)$ satisfies
\begin{align*}
\lim_{t \to \infty} \Bigg[
\left\| u(\cdot, t) - \left( \frac{a}{b} \right)^{\gamma - 1} \right\|_{L^\infty}
+\left\| v(\cdot, t) - \frac{\mu_1}{\lambda_1} \left( \frac{a}{b} \right)^{\gamma - 1} \right\|_{L^\infty}
+ \left\| w(\cdot, t) - \frac{\mu_2}{\lambda_2} \left( \frac{a}{b} \right)^{\gamma - 1} \right\|_{L^\infty}
\Bigg]
= 0
\end{align*}
if one of the following assumptions holds:
\begin{enumerate}
\item[\rm(a)] $\gamma = k + 1$, $b + \chi_2 \mu_2 - \chi_1 \mu_1 - H > 0$;
\item[\rm(b)] $\gamma \neq k + 1$, $\chi_2 \mu_2 = \chi_1 \mu_1$, $\lambda_1 = \lambda_2$,
\end{enumerate}
where
\begin{align*}
H := \min\Big\{ &\frac{1}{\lambda_1} \left[ \vert \chi_1 \mu_1 \lambda_1 - \chi_2 \mu_2 \lambda_2 \vert + \chi_2 \mu_2 \vert \lambda_1 - \lambda_2 \vert \right], \\
&\frac{1}{\lambda_2} \left[ \vert \chi_1 \mu_1 \lambda_1 - \chi_2 \mu_2 \lambda_2 \vert + \chi_1 \mu_1 \vert \lambda_1 - \lambda_2 \vert \right] \Big\}.
\end{align*}
\end{theorem}

\begin{theorem}[Spreading Speed]\label{thm:mytheorem1.4}
Let $\gamma > 1$, $\alpha \in ( \frac{1}{2}, 1)$, $N>2\alpha$, $k\geq1$, $\chi_1, \chi_2, a, b>0$, $u_0 \in C_{unif}^b \left( \mathbb{R}^N \right)$ and $\inf\limits_{x \in \mathbb{R}^N} u_0(x) > 0$.
For $u_0$ with nonempty compact support and $0 < \varepsilon < \frac{a}{N + 2\alpha}$, then
\[
\liminf_{t \to \infty} \inf_{\substack{|x| \le e^{\left( \frac{a}{N + 2\alpha} - \varepsilon \right) t}}} u(x, t) > 0
\]
if one of the following assumptions holds:
\begin{enumerate}
\item[\rm(a)] $\gamma = k + 1,\ b + \chi_2 \mu_2 - \chi_1 \mu_1 - M > 0$;
\item[\rm(b)] $\gamma \neq k + 1,\ \chi_2 \mu_2 = \chi_1 \mu_1,\ \lambda_1 = \lambda_2$.
\end{enumerate}
Furthermore, for $u_0 \in X_0$ and $\varepsilon > 0$, then
\[
\begin{cases}
\lim\limits_{t \to \infty} \sup\limits_{|x| \geq e^{\left( \frac{a+ MC_0^k }{N + 2\alpha} + \varepsilon \right) t}} u(x, t) = 0, & if ~ \mathrm{(a)} ~holds, \\
\lim\limits_{t \to \infty} \sup\limits_{|x| \geq e^{\left( \frac{a}{N + 2\alpha} + \varepsilon \right) t}} u(x, t) = 0, & if ~ \mathrm{(b)} ~holds,
\end{cases}
\]
where $X_0$ is defined as
\[
X_0 := \left\{ u \in C_{unif}^b \left( \mathbb{R}^N \right), \inf\limits_{x \in \mathbb{R}^N} u_0 > 0 \text{ and } u_0(x) \leq C^* |x|^{-N-2\alpha} \text{ for all } x \in \mathbb{R}^N \right\}.
\]
\end{theorem}

\noindent\textbf{Remark 1.3} This result covers and extends the spreading properties established by Jiang et al. \cite{JLL} with logistic source $\gamma = 2$ and linear productions $k = 1$.

\noindent\textbf{Remark 1.4} When $\gamma = k + 1$, the upper bound of the propagation speed is $\frac{a+ M C_0^{k}}{N + 2\alpha}$. Given that $C_0 > 1$, it can be deduced that the upper bound of the propagation speed has a monotonically increasing relationship with $k$. In particular, since $M$ characterizes the balance coefficient between attractive and repulsive effects, when $M = 0$ (that is, $\chi_1 \mu_1 = \chi_2 \mu_2$ and $\lambda_1 = \lambda_2$), the spreading speed is exactly $\frac{a}{N + 2\alpha}$.

\noindent\textbf{Remark 1.5} Similar to the proof process of \cref{thm:mytheorem1.4}, we can obtain the persistence of the solution. Let $\gamma > 1$, $\alpha \in ( \frac{1}{2}, 1)$, $k\geq1$, $\chi_1, \chi_2, a, b>0$, $u_0 \in C_{unif}^b \left( \mathbb{R}^N \right)$ and $\inf\limits_{x \in \mathbb{R}^N} u_0(x) > 0$. Then there exists a time $\tilde{T}$ such that
\[
m \leq u(x, t) \leq M
\]

\noindent if one of the following assumptions holds:
\begin{enumerate}
\item[(a)] $\gamma = k + 1,\ b + \chi_2 \mu_2 - \chi_1 \mu_1 - M > 0$;
\item[(b)] $\gamma \neq k + 1,\ \chi_2 \mu_2 = \chi_1 \mu_1,\ \lambda_1 = \lambda_2$
\end{enumerate}
for all $x \in \mathbb{R}^N$ and $t \geq \tilde{T}$, where $M$ is given in \cref{thm:mytheorem1.2}.

First, employing the Schauder's fixed point theorem, we can deduce the global boundedness of classical solutions. Using appropriate estimates for $\chi_1 \mu_1 \lambda_1 - \chi_2 \mu_2 \lambda_2$ and the comparison principle, we can obtain the global asymptotic stability of the solution via the upper-lower solution method. Drawing upon the methodology developed in \cite{JLL}, we establish the proof of the lower bounds of the spreading speed by employing several essential estimates concerning $u$, $v$ and $w$. For the upper bounds of spreading speed, the argument is completed through the construction of appropriate supersolutions combined with applications of the comparison principle and fractional heat kernel estimates.

This paper is organized as follows. In Section 2, we prove the global existence and boundedness of classical solutions. The asymptotic behavior of \eqref{1.1} is investigated in Section 3. Finally, Section 4 is devoted to studying the spreading speed of \eqref{1.1}.



\section{ Global Boundedness }
Building on the local existence theorem for classical solutions (Proposition \ref{pro1.1}), we extend the analysis to establish the global boundedness as stated in \cref{thm:mytheorem1.2}.

First, we consider the following Cauchy problem
\begin{equation}\label{1.0.1}
\left\{
\begin{aligned}
&u_t + (-\Delta)^\alpha u = 0, \quad &&x \in \mathbb{R}^N, \, t > 0, \\
&u(x, 0) = u_0(x), \quad &&x \in \mathbb{R}^N.
\end{aligned}
\right.
\end{equation}
The solutions of \eqref{1.0.1} can be written as $u(t) = K_t^\alpha (x) * u_0(x)$. Here, \( K_t^\alpha (x) \) is a fractional heat kernel, which is denoted by
\[
K_t^\alpha (x) := t^{-\frac{N}{2\alpha}} K^\alpha \left( t^{-\frac{1}{2\alpha}} x \right),
\]
where $K^\alpha (x) := (2\pi)^{-N} \int_{\mathbb{R}^N} e^{i \xi \cdot x} e^{-|\xi|^{2\alpha}} d\xi$ \cite{ZZL,DH}. The operator \(-(-\Delta)^\alpha + I \) generates an analytic semigroup \( \{ T(t) \}_{t \geq 0} \) and
\begin{equation*}
\begin{aligned}
T(t)u = e^{-t} (K_t^\alpha (x) * u)(x) = \int_{\mathbb{R}^N} e^{-t} K_t^\alpha (x - y) u(y) dy
\end{aligned}
\end{equation*}
for every \( u \in X \), \( x \in \mathbb{R}^N \) and \( t \geq 0 \), where Banach space \( X = C_{unif}^b (\mathbb{R}^N) \) or \( X = L^p (\mathbb{R}^N) \).

For \( (\beta \in [0, \infty)) \), define the space \( X^\beta = \text{Dom}(((-\Delta)^\alpha + I)^\beta) \). Then we have the following inequalities
\begin{equation}\label{0.0.1}
\begin{aligned}
\| (T(t) - I)u \|_{L^p(\mathbb{R}^N)} \leq C_{\beta} t^{\beta} \| u \|_{X^{\beta}} \quad \text{for } 0 < \beta \leq 1,\ u \in X^{\beta}, \\
\end{aligned}
\end{equation}
\begin{equation}\label{0.0.2}
\begin{aligned}
\| ((-\Delta)^{\alpha} + I)^{\beta} T(t)u \|_{L^q(\mathbb{R}^N)} \leq M_{\beta} t^{-\beta- \left( \frac{1}{p} - \frac{1}{q} \right) \frac{N}{2\alpha}} e^{-t} \| u \|_{L^p(\mathbb{R}^N)}
\end{aligned}
\end{equation}
for \(1 \leq p \leq q \leq \infty\), \(t > 0\) and \(\beta \geq 0\), where \(C_{\alpha}\), \(C_{\beta}\) and \(M_{\beta}\) are constants depending only on \(\alpha\), \(\beta\), \(p\), \(q\) and \(N\).
We now present several lemmas of the fractional heat kernel theory.
\begin{lemma}[\cite{ZZL}, Lemma 4.2]
For each \(t > 0\), the operator \(T(t)\nabla \cdot\) can be uniquely extended to a bounded linear operator acting on the space \((C_{unif}^{b}(\mathbb{R}^N))^N\), such that
\[
\| T(t)\nabla \cdot u \|_{L^\infty} \leq C_1 t^{-\frac{1}{2\alpha}} e^{-t} \| u \|_{L^\infty}
\]
for all $u \in (C_{unif}^{b}(\mathbb{R}^N))^N$ and $t > 0$, where \(C_1\) is a positive constant depending exclusively on $\alpha$ and $N$.
\label{lemma:Lemma 2.1}
\end{lemma}

\begin{lemma}\label{lemma:Lemma 2.5}
For every \( u \in C_{unif}^b(\mathbb{R}^N) \), we can deduce that
\begin{align*}
\bigl\lVert \nabla \mu_i (\Delta - \lambda_i I)^{-1} u \bigr\rVert_{L^\infty}
\leq \frac{\sqrt{N} \, \mu_i}{\sqrt{\lambda_i}} \bigl\lVert u^{k} \bigr\rVert_{L^\infty}, \, i = 1, 2.
\end{align*}
\end{lemma}
\begin{proof}
Similar to the argument in the proof of Lemma 3.3 in \cite{BW1}, we omit the details here.
\end{proof}

\begin{definition}[\cite{PR}]
For any real-valued function \( f \) on \( \mathbb{R}^N \), define
\[
K_f^{2\alpha}(r) := \sup_{x \in \mathbb{R}^N} \int_{B_r(x)} \frac{|f(y)|}{|x - y|^{N + 1 - 2\alpha}} dy, \text{ for } r > 0,
\]
where \( B_r(x) \) denotes the open ball centered at \( x \in \mathbb{R}^N \) with radius \( r \). Then \( f \) is said to belong to the Kato class \( K^{2\alpha - 1} \) if \( \lim_{r \to 0} K_f^{2\alpha}(r) = 0 \).
\label{def: Dedinition 1} 
\end{definition}

\noindent \textbf{Proof of Theorem 1.1.} We introduce a linear normed space \( Q = C^b_{\textit{unif}}(\mathbb{R}^N \times [0,T]) \) equipped with the norm
\[
\| u \|_Q := \sum_{k=1}^\infty \frac{1}{2^k} \| u \|_{L^\infty([-r,r] \times [0,T])},
\]
where \( T \) is a positive real number.
We define a subset \( Q_0 \) of \( Q \) as
\[
Q_0 := \left\{ u \in Q \,\bigg|\, \ 0 \leq u(x,t) \leq C_0, u(\cdot, 0) = u_0 \right\}.
\]
For any $u_0 \in Q_0$, we have $\|u\|_Q \leq C_0$. Given that the initial data $u_0 \in Q_0$ satisfies $\inf_{x \in \mathbb{R}^N} u_0(x) > 0$, the classical solution $(u, v, w)$ to \eqref{1.1} exists for $t \in [0, T_{\text{max}})$. For any $x \in \mathbb{R}^N$, we have
\begin{equation*}
\begin{split}
u_t &= -(-\Delta)^\alpha u + \nabla(\chi_2 w - \chi_1 v) \nabla u \\
&\quad + u\left( a + \chi_2 \lambda_2 w - \chi_1 \lambda_1 v - b u^{\gamma-1} + (\chi_1 \mu_1 - \chi_2 \mu_2) u^k \right),
\end{split}
\end{equation*}
where \( v, w \) satisfies $v = (\lambda_1 I - \Delta)^{-1} \mu_1 u^k,~~w = (\lambda_2 I - \Delta)^{-1} \mu_2 u^k $.

Denote by \( U(x,t,u):=U(x,t) \) the solution of the following equation
\begin{equation*}
\begin{cases}
\begin{aligned}
U_t &= -(-\Delta)^\alpha U + \nabla(\chi_2 w - \chi_1 v) \nabla U \\
&\quad + U\bigl( a + \chi_2 \lambda_2 w - \chi_1 \lambda_1 v - b U^{\gamma-1} + (\chi_1 \mu_1 - \chi_2 \mu_2) U^k \bigr),
\end{aligned} \\[6pt]
U(\cdot, 0) = u_0.
\end{cases}
\end{equation*}
For every \( u \in Q_0 \), we derive that
\begin{align*}
(\chi_2\lambda_2w - \chi_1\lambda_1v)(x,t) &= \int_0^\infty \int_{\mathbb{R}^N}
\frac{e^{-\frac{|x-z|^2}{4s}}}{(4\pi s)^{\frac{N}{2}}}
\left(
\chi_2\lambda_2\mu_2 e^{-\lambda_2 s}
- \chi_1\lambda_1\mu_1 e^{-\lambda_2 s}
\right. \\
&\quad \left.
+ \chi_1\lambda_1\mu_1 e^{-\lambda_2 s}
- \chi_1\lambda_1\mu_1 e^{-\lambda_1 s}
\right)
u^k(z,t) \, dz \, ds \\
&\leq (\chi_2\lambda_2\mu_2 - \chi_1\lambda_1\mu_1)_+ C_0^k
\int_0^\infty \int_{\mathbb{R}^N}
\frac{e^{-\frac{|x-z|^2}{4s}}}{(4\pi s)^{\frac{N}{2}}}
e^{-\lambda_2 s} \, dz \, ds \\
&\quad + \chi_1\lambda_1\mu_1 C_0^k
\int_0^\infty \int_{\mathbb{R}^N}
\frac{e^{-\frac{|x-z|^2}{4s}}}{(4\pi s)^{\frac{N}{2}}}
\left( e^{-\lambda_2 s} - e^{-\lambda_1 s} \right)_+ \, dz \, ds \\
&\leq \frac{C_0^k}{\lambda_2}
\left[
(\chi_2\lambda_2\mu_2 - \chi_1\lambda_1\mu_1)_+
+ \chi_1\lambda_1\mu_1 (\lambda_1 - \lambda_2)_+
\right].
\end{align*}
Similarly, we have
\begin{align*}
(\chi_2\lambda_2 w - \chi_1\lambda_1 v)(x,t) \leq \frac{C_0^k}{\lambda_1} \left[ (\chi_2\lambda_2\mu_2 - \chi_1\lambda_1\mu_1)_+ + \chi_2\lambda_2\mu_2 (\lambda_1 - \lambda_2)_+ \right].
\end{align*}
\noindent Hence, we have \( (\chi_2\lambda_2 w - \chi_1\lambda_1v)(x,t) \leq M C_0^k \), where \( M \) is given in \cref{thm:mytheorem1.2}.
\noindent Then for any \( u \in Q_0 \), we obtain
\begin{equation*}
\begin{aligned}
U_t \leq -(-\Delta)^\alpha U + \nabla(\chi_2 w - \chi_1v) \nabla U + U\left( a + M C_0^k - b U^{\gamma-1} + (\chi_1\mu_1 - \chi_2\mu_2) U^k \right).
\end{aligned}
\end{equation*}

Next, we will show that \( U(x,t) \in Q_0 \). We divide the proof into four cases.

\noindent{\textbf{Case 1}. \( \gamma \geq k+1 \)}.

Let \( C_0 := \max\left\{1, \| u_0 \|_{L^\infty}, \left( \frac{a}{b + \chi_2\mu_2 - \chi_1\mu_1 - M} \right)^{\frac{1}{k}} \right\} \). Since \( b + \chi_2\mu_2 - \chi_1\mu_1 - M > 0 \), we have
\[
a + (M + \chi_1\mu_1 - \chi_2\mu_2)C_0^k - b C_0^{\gamma - 1} \leq a + (M + \chi_1\mu_1 - \chi_2\mu_2 - b)C_0^k \leq 0.
\]
For any $u \in Q_0$, it follows that
\[
U(x,t) \leq C_0
\]
for all $t \in [0,T]$, $x \in \mathbb{R}^N$.\\
\textbf{Case 2.} \( \gamma \leq k + 1 \).

Let \( C_0 := \max\left\{ 1, \| u_0 \|_{L^\infty}, \left( \frac{a}{b} \right)^{\frac{1}{\gamma - 1}} \right\} \). Due to \( \chi_2\lambda_2\mu_2 > \chi_1\lambda_1\mu_1 \) and \( \lambda_1 > \lambda_2 \), we can derive that \( M = \chi_2\mu_2 - \chi_1\mu_1 \) and
\[
a + (M + \chi_1\mu_1 - \chi_2\mu_2)C_0^k - bC_0^{\gamma - 1} = a - bC_0^{\gamma - 1} \leq 0.
\]
For any $u \in Q_0$, we can deduce
\[
U(x,t) \leq C_0
\]
for all $t \in [0,T]$, $x \in \mathbb{R}^N$.

\noindent \(\textbf{Case 3.}\) \( \gamma \leq k + 1 \).

Let \( C_0 := \max\left\{ 1, \| u_0 \|_{L^\infty}, C_* \right\} \), where \( C_* \leq \left( \frac{b - a}{M + \chi_1\mu_1} \right)^{\frac{1}{k}} \). Since \( \| u_0 \|_{L^\infty} \leq \left( \frac{b - a}{M + \chi_1\mu_1} \right)^{\frac{1}{k}} \) and \( b > a + M + \chi_1\mu_1 \), we have
\[
a + (M + \chi_1\mu_1)C_0^k - bC_0^{\gamma - 1} \leq a + (M + \chi_1\mu_1)C_0^k - b \leq 0.
\]
For any $u \in Q_0$, using the comparison principle of parabolic equations, we have
\[
U(x,t) \leq C_0
\]
for all $t \in [0,T]$, $x \in \mathbb{R}^N$.

\noindent \(\textbf{Case 4.}\) \( \gamma \neq k + 1 \).

Let \( C_0 := \max\left\{ 1, \| u_0 \|_{L^\infty}, \left( \frac{a}{b} \right)^{\frac{1}{\gamma - 1}} \right\} \). Due to \( \chi_2\mu_2 = \chi_1\mu_1 \), \( \lambda_1 = \lambda_2 \), we have \( M = 0 \) and
\[
a + (M + \chi_1\mu_1 - \chi_2\mu_2)C_0^k - bC_0^{\gamma - 1} = a - bC_0^{\gamma - 1} \leq 0.
\]
For any $u \in Q_0$, we conclude from the comparison principle of parabolic equations that
\[
U(x,t) \leq C_0
\]
for all $t \in [0,T]$, $x \in \mathbb{R}^N$.
Thus, for any \( u \in Q_0 \), we have \( U(x,t) \in Q_0 \).

Next, we aim to prove that the mapping \( u \mapsto U(x,t) \in Q_0 \) is compact and continuous. To this end, we shall divide the proof into two steps.

\noindent \textbf{Step 1.} The mapping \( u \mapsto U(x,t) \in Q_0 \) is compact.

Let \( \{ u_n \}_{n \geq 1} \subset Q_0 \). For any \( n \geq 1 \), the mapping \( u \mapsto U(x,t) \in Q_0 \) satisfies
\begin{equation*}
\begin{cases}
\begin{aligned}
U_t(x,t,u_n) &= -(-\Delta)^\alpha U(x,t,u_n)
- \chi_1 \nabla \cdot \left( U(x,t,u_n) \nabla (\Delta - \lambda_1 I)^{-1}\mu_1 U^k(x,t,u_n) \right) \\
&\quad + \chi_2 \nabla \cdot \left( U(x,t,u_n) \nabla (\Delta - \lambda_2 I)^{-1}\mu_2 U^k(x,t,u_n) \right)\\
&\quad + U(x,t,u_n) \left( a - b U^{\gamma - 1}(x,t,u_n) \right),
\end{aligned} \\[6pt]
U(\cdot, 0, u_n) = u_0.
\end{cases}
\end{equation*}
Then we have
\begin{equation*}
\begin{aligned}
U(x,t,u_n) &= T(t)u_0 + \chi_1 \int_0^t T(t - \tau) \nabla \cdot \left( U(x,\tau, u_n) \nabla (\Delta - \lambda_1 I)^{-1} \mu_1 U^k(x,\tau, u_n) \right) d\tau \\
&\quad - \chi_2 \int_0^t T(t - \tau) \nabla \cdot \left(U(x,\tau, u_n) \nabla (\Delta - \lambda_2 I)^{-1} \mu_2 U^k(x,\tau, u_n) \right) d\tau \\
&\quad + \int_0^t (a + 1) T(t - \tau) U(x,\tau, u_n) d\tau - b \int_0^t T(t - \tau) U^{\gamma - 1}(x,\tau, u_n) d\tau \\
&:= T(t)u_0 + L_1^n(t) - L_2^n(t) + L_3^n(t).
\end{aligned}
\end{equation*}

Firstly, for any \( 0 < t_1 < T \), we will prove the mapping \( [t_1, T] \ni t \mapsto U(x,t,u_n) \in X^\beta \) is uniformly bounded and equicontinuous if \( \frac{1}{2} < \alpha < 1 \) and \( 0 < \beta < \frac{1}{2} - \frac{1}{4\alpha} \).

For every \( 0 < t < T \), we have
\[
\| U(\cdot, t, u_n) \|_{X^\beta} \leq \| T(t)u_0 \|_{X^\beta} + \| L_1^n(t) \|_{X^\beta} + \| L_2^n(t) \|_{X^\beta} + \| L_3^n(t) \|_{X^\beta}.
\]
Applying \eqref{0.0.2}, \cref{lemma:Lemma 2.1} and \cref{lemma:Lemma 2.5}, we get
\begin{align*}
\| L_1^n(t) \|_{X^\beta}
&= \chi_1\int_0^t \| \left( (-\Delta)^\alpha + I \right)^{\beta} T(t - \tau) \nabla \cdot \left( U(\tau) \nabla (\Delta - \lambda_1 I)^{-1} \mu_1 U^k(\tau) \right) \|_{L^\infty} d\tau \\
&\leq \chi_1C_1 M_{\beta} C_0 \int_0^t (t - \tau)^{-\beta - \frac{1}{2\alpha}} e^{-(t - \tau)} \| U(\tau) \nabla (\Delta - \lambda_1 I)^{-1} \mu_1 U^k(\tau) \|_{L^\infty} d\tau \\
&\leq \frac{C_1 M_{\beta} C_0^{k+1} \chi_1 \mu_1 \sqrt{N}}{\sqrt{\lambda_1}} \Gamma \left( 1 - \beta - \frac{1}{2\alpha} \right).
\end{align*}
Similar to \( L_1^n(t) \), we conclude that
\begin{align*}
\| L_2^n(t) \|_{X^\beta} &\leq \frac{C_1 M_{\beta} C_0^{k+1} \chi_2 \mu_2 \sqrt{N}}{\sqrt{\lambda_2}} \Gamma \left( 1 - \beta - \frac{1}{2\alpha} \right)
\end{align*}
and
\begin{align*}
\| L_3^n(t) \|_{X^\beta}
&\leq M_\beta C_0 \left[ (a + 1) + b C_0^{\gamma - 1} \right] \Gamma (1 - \beta).
\end{align*}
Since \( \frac{1}{2} < \alpha < 1 \) and \( 0 < \beta < \frac{1}{2} - \frac{1}{4\alpha} \), we know that \( \Gamma(1 - \beta) \) and \( \Gamma\left( 1 - \beta - \frac{1}{2\alpha} \right) \) are bounded.
From an elementary caculation, we can deduce that
\[
\| T(t) u_0 \|_{X^\beta} \leq M_\beta C_0 {t_1}^{-\beta} e^{t_1}.
\]
Then, combining the above inequalities, we conclude that
\begin{equation*}
\begin{aligned}
\sup_{t_1 \leq t \leq T} \| U(\cdot, t, u_n) \|_{X^\beta} &\leq M_\beta C_0^{k+1}C_1 \sqrt{N} \left( \frac{\chi_1 \mu_1}{\sqrt{\lambda_1}} + \frac{\chi_2 \mu_2}{\sqrt{\lambda_2}} \right) \Gamma \left( 1 - \beta - \frac{1}{2\alpha} \right) \\
&\quad + M_\beta C_0 \left[ (a + 1) + b C_0^{\gamma - 1} \right] \Gamma (1 - \beta) + M_\beta C_0 {t_1}^{-\beta} e^{-t_1}.
\end{aligned}
\end{equation*}
Next, let \( t_1 \leq t + h \leq T \). According to \eqref{0.0.1}, \cref{lemma:Lemma 2.1} and \cref{lemma:Lemma 2.5}, we have
\begin{align*}
\| T(t + h)u_0 - T(t)u_0 \|_{X^\beta} &\leq C(\beta) h^\beta t_1^{-\beta} e^{-t_1} \| u_0 \|_{L^\infty}.
\end{align*}
Applying \eqref{0.0.1}, \eqref{0.0.2}, \cref{lemma:Lemma 2.1} and \cref{lemma:Lemma 2.5}, we have
\begin{align}\label{2.8}
&\quad \| L_1^n(t + h) - L_1^n(t) \|_{X^\beta} \nonumber\\
&= \chi_1\int_0^t \| (T(h) - I) T(t - \tau) \nabla \cdot \left( U(\tau) \nabla (\Delta - \lambda_1 I)^{-1} \mu_1 U^k(\tau) \right) \|_{X^\beta} d\tau \nonumber\\
&\quad + \chi_1 \int_t^{t + h} \| T(t + h - \tau) \nabla \cdot \left( U(\tau) \nabla (\Delta - \lambda_1 I)^{-1} \mu_1 U^k(\tau) \right) \|_{X^\beta} d\tau \nonumber\\
&\leq C_\beta h^\beta \chi_1\int_0^t \| \left( (-\Delta)^\alpha - I \right)^{2\beta} T(t - \tau) \nabla \cdot \left( U(\tau) \nabla (\Delta - \lambda_1 I)^{-1} \mu_1 U^k(\tau) \right) \|_{L^\infty} d\tau \nonumber\\
&\quad + C_1 M_\beta \chi_1\int_t^{t + h} (t + h - \tau)^{-\beta - \frac{1}{2\alpha}} e^{-(t + h - \tau)} \| U(\tau) \nabla (\Delta - \lambda_1 I)^{-1} \mu_1 U^k(\tau) \|_{L^\infty} d\tau \nonumber\\
&\leq \frac{C_1 M_\beta \chi_1\mu_1 \sqrt{N}}{\sqrt{\lambda_1}} C_0^{k + 1} \left[C_\beta h^\beta \Gamma \left( 1 - 2\beta - \frac{1}{2\alpha} \right) + \frac {h^{1-\beta-\frac{1}{2\alpha}}}{1-\beta-\frac{1}{2\alpha}} \right].
\end{align}
Similarly, the same procedure to \eqref{2.8} yields
\[
\| L_2^n(t + h) - L_2^n(t) \|_{X^\beta} \leq \frac{C_1 M_\beta \chi_2 \mu_2 \sqrt{N}}{\sqrt{\lambda_2}} C_0^{k + 1} \left[C_\beta h^\beta \Gamma \left( 1 - 2\beta - \frac{1}{2\alpha} \right) + \frac {h^{1-\beta-\frac{1}{2\alpha}}}{1-\beta-\frac{1}{2\alpha}} \right]
\]
and
\begin{align*}
\| L_3^n(t+h) - L_3^n(t) \|_{X^\beta}
\le M_\beta C_0 \bigl[ (a+1) + b C_0^{\gamma-1} \bigr] \Bigl[ C_\beta h^\beta \Gamma(1-2\beta) + \frac{h^{1-\beta}}{1-\beta} \Bigr].
\end{align*}
Consequently, there is a constant \( C \) for which
\begin{equation*}
\begin{aligned}
\| U(\cdot, t + h, u_n) - U(\cdot, t, u_n) \|_{X^\beta} \leq C \left( h^\beta + h^{1 - \beta} + h^{1-\beta - \frac{1}{2\alpha}} \right).
\end{aligned}
\end{equation*}
Using the Arzel\`{a}-Ascoli theorem together with the continuous embedding $X^\beta \hookrightarrow C^\nu$ for \( 0 \leq \nu \leq 2\beta \), we conclude that there exists a function \( U \in C^{2,1}(\mathbb{R}^N \times (0, T]) \) and a subsequence \( \{ u_{n_j} \}_{j \geq 1} \) of \( \{ u_n \}_{n \geq 1} \) such that \( U(\cdot, t, u_{n_j}) \) converges locally uniformly to \( U \in C^{2,1}(\mathbb{R}^N \times (0, T]) \).

Secondly, we prove that the function \( \Phi(u) \) is continuous at the initial time \( t = 0 \).

Fix a positive number \( \varepsilon > 0 \). Then there exists \( 0 < t_\varepsilon < T \) such that
\begin{align*}
\| T(t)u_0 - u_0 \|_{L^\infty} &+ C_1 M_\beta C_0^{k+1} \sqrt{N} \left( \frac{\chi_1 \mu_1}{\sqrt{\lambda_1}} + \frac{\chi_2 \mu_2}{\sqrt{\lambda_2}} \right) \Gamma \left( 1 - \beta - \frac{1}{2\alpha} \right) \\
&+ M_\beta C_0 \left( (a + 1) + b C_0^{\gamma - 1} \right)\Gamma(1-\beta) + 2 M_\beta C_0^{-\beta} e^{-t_1} < \varepsilon
\end{align*}
for any \( t \leq t_\varepsilon \). Choosing \( \beta = 0 \), we get
\[
\begin{aligned}
\| U(x,t, u_{n_j}) - u_0 \|_{L^\infty} &= \| U(x,t, u_{n_j}) - T(t)u_0 + T(t)u_0 - u_0 \|_{L^\infty} \\
&\leq \| U(x,t, u_{n_j}) \|_{L^\infty} + \| T(t)u_0 \|_{L^\infty} + \| T(t)u_0 - u_0 \|_{L^\infty} \leq \varepsilon
\end{aligned}
\]
for any \( t \leq t_\varepsilon \). Let \( n \to \infty \),
\[
\| U(\cdot, t) - u_0 \|_{L^\infty} \leq \varepsilon
\]
for any \( t \leq t_\varepsilon \).

\noindent \textbf{Step 2.} The mapping \( u \mapsto U(x,t,u) \in Q_0\) is continuous.

Suppose the mapping \( u \mapsto U(x,t,u) \in Q_0 \) is discontinuous. Assume that \( \{ u_n \}_{n \geq 1} \subset Q_0 \) and \( u \in Q_0 \) satisfy \( \| u_n - u \|_{L^\infty} \to 0 \), \( n \to 0 \). Then there exist \( \delta > 0 \) and a subsequence \( \{ u_{n_j} \}_{j \geq 1} \) of \( \{ u_n \}_{n \geq 1} \) satisfying
\[
\| U(x,t, u_{n_j}) - U(x,t, u) \|_{L^\infty} \geq \delta
\]
for all $n \geq 1$.
As shown in Step 1, the mapping \( u \mapsto U(x,t,u) \in Q_0 \) is compact. Thus, there exist a subsequence \( \{ u_{n_j} \}_{j \geq 1} \) of \( \{ u_n \}_{n \geq 1} \) and \( U(x,t,u) \in Q \cap C^{2,1}(\mathbb{R}^N \times (0,T]) \) such that in \( C_{loc}^{2,1}(\mathbb{R}^N \times (0,T]) \), as \( n \to \infty \), \( U(x,t,u_{n_j}) \) converges to \( U(x,t,u) \). However, for each \( n \geq 1 \), we have
\begin{align}\label{2.11}
U_t(x,t,u_{n_j}) = &-(-\Delta)^\alpha U(x,t,u_{n_j}) - \chi_1 \nabla \cdot \left( U(x,t,u_{n_j}) \nabla (\Delta - \lambda_1 I)\mu_1 U^k(x,t,u_{n_j}) \right) \nonumber\\
&+ \chi_2 \nabla \cdot \left( U(x,t,u_{n_j}) \nabla (\Delta - \lambda_2 I)\mu_2 U^k(x,t,u_{n_j}) \right)\nonumber\\
&+ U(x,t,u_{n_j}) \left( a - b U^{\gamma - 1}(x,t,u_{n_j}) \right).
\end{align}
Let \( n \to \infty \) in \eqref{2.11}, it follows that
\begin{align}\label{2.12}
\begin{cases}
U_t(\cdot,t) = -(-\Delta)^\alpha U(\cdot,t)
- \chi_1 \nabla \cdot \left( U(\cdot,t) \nabla (\Delta - \lambda_1 I)^{-1} \mu_1 U^k(\cdot,t) \right) \\
~\qquad \qquad + \chi_2 \nabla \cdot \left( U(\cdot,t) \nabla (\Delta - \lambda_2 I)^{-1} \mu_2 U^k(\cdot,t) \right)
+ U(\cdot,t) \left( a - b U^{\gamma - 1}(\cdot,t) \right), \\
U(\cdot, 0) = u_0.
\end{cases}
\end{align}
Since \( U(\cdot,\cdot, u) \) is the unique bounded classical solution of \eqref{2.12}, we have \( U(\cdot,\cdot) = U(\cdot,\cdot, u) \). Then,
\[
\lim_{n \to \infty} \| U(\cdot,\cdot, u_{n_j}) - U(\cdot,\cdot, u) \|_Q = 0,
\]
which is a contradiction.

Combining with Step 1 and Step 2, we conclude that the mapping \( u \mapsto U(x,t,u) \in Q_0 \) is both compact and continuous. By using the Schauder fixed point theorem, we can know that there exists a fixed point \( u^* \). In light of the local existence of solutions, it holds that \( T_{\text{max}} \geq T \) with \( u(\cdot,t, u_0) = u^* \). In view of thearbitrariness of \( T > 0 \), this implies \( T_{\text{max}} = \infty \). This completes the proof of \cref{thm:mytheorem1.2}.

\section{Asymptotic Behavior}
In this section, we analyze the long time behavior of solutions to \eqref{1.1} in two distinct cases : \( \gamma \neq k + 1 \) and \( \gamma = k + 1 \). Prior to proving \cref{thm:mytheorem1.3}, we first present the following lemma.

\begin{lemma}
Under the assumptions of \cref{thm:mytheorem1.2}, if $\displaystyle \inf_{x \in \mathbb{R}^N} u_0(x) > 0$, then $\displaystyle \inf_{x \in \mathbb{R}^N} u(x,t)$ remains positive.
\label{lemma:Lemma 4.1}
\end{lemma}
\begin{proof}
Let $Z := \chi_1 \lambda_1 \sup_{x \in \mathbb{R}^N} v(x,t)$. Since
\[ u_t = -(-\Delta)^\alpha u + \nabla (\chi_2 w - \chi_1 v) \nabla u + u \left( a + (\chi_2 \lambda_2 w - \chi_1 \lambda_1 v) - b u^{\gamma - 1} + (\chi_1 \mu_1 - \chi_2 \mu_2) u^k \right), \]
we can obtain
\[ u_t \geq -(-\Delta)^\alpha u + \nabla (\chi_2 w - \chi_1 v) \nabla u + u \left( a - Z - b u^{\gamma - 1} + (\chi_1 \mu_1 - \chi_2 \mu_2) u^k \right). \]

Let \( \bar{w} \) stand for the solution to the ODE given below
\[
\begin{cases}
\bar{w}_t = \bar{w} \left( a - Z - b \bar{w}^{\gamma - 1} + (\chi_1 \mu_1 - \chi_2 \mu_2) \bar{w}^k \right), \\
\bar{w}(0) = \inf_{x \in \mathbb{R}^N} u_0(x).
\end{cases}
\]
According to the comparison principle, it can be obtained that $u(x,t) \geq \bar{w}(t)$ for all $x \in \mathbb{R}^N$ and $t>0$.
Since
\[
\begin{cases}
b + \chi_2 \mu_2 - \chi_1 \mu_1 > 0, & \text{if } \gamma = k + 1, \\
b > 0, & \text{if } \gamma \neq k + 1
\end{cases}
\]
and \( \inf_{x \in \mathbb{R}^N} u_0(x) > 0 \), we can get \( \bar{w}(t) > 0 \), which implies that \( 0 < \bar{w}(t) \leq \inf_{x \in \mathbb{R}^N} u(x,t) \) for every \( t \geq 0 \).
\end{proof}
Now, we define \( \bar{u} = \limsup_{t \to \infty} \inf_{x \in \mathbb{R}^N} u(x,t) \), \( \underline{u} = \liminf_{t \to \infty} \inf_{x \in \mathbb{R}^N} u(x,t) \). Then for any \( \varepsilon > 0 \), we can choose \( \tilde{T}_\varepsilon > 0 \) satisfying
\[
\mu_1 (\underline{u} - \varepsilon)^k \leq \lambda_1 v(x,t) \leq \mu_1 (\bar{u} + \varepsilon)^k,~~\mu_2 (\underline{u} - \varepsilon)^k \leq \lambda_2 w(x,t) \leq \mu_2 (\bar{u} + \varepsilon)^k
\]
for all $t \geq \tilde{T}_\varepsilon$.

Now we turn to analyzing the asymptotic behavior of the global classical solution to \eqref{1.1} with \( \inf_{x \in \mathbb{R}^N} u_0(x) > 0 \).\\
\(\textbf{Proof of Theorem 1.2.}\) We divide the following proof into two cases.

\noindent When \( \gamma = k + 1 \).

(1) Assume \( b + \chi_2\mu_2 - \chi_1\mu_1 - \frac{1}{\lambda_2} \left[ |\chi_2\lambda_2\mu_2 - \chi_1\lambda_1\mu_1| + \chi_1\mu_1|\lambda_1 - \lambda_2| \right] > 0. \)

For any \( x \in \mathbb{R}^N, t \geq \tilde{T}_\varepsilon \), by simple calculation, we obtain that
\begin{align*}
(\chi_2\lambda_2 w - \chi_1\lambda_1 v)(x,t) &= (\chi_2\lambda_2\mu_2 - \chi_1\lambda_1\mu_1) \int_0^\infty \int_{\mathbb{R}^N} e^{-\lambda_2 s} \frac{e^{-\frac{|x - z|^2}{4s}}}{(4\pi s)^{\frac{N}{2}}} u^k(z,t) dz ds \\
&\quad + \chi_1\lambda_1\mu_1 \int_0^\infty \int_{\mathbb{R}^N} (e^{-\lambda_2 s} - e^{-\lambda_1 s}) \frac{e^{-\frac{|x - z|^2}{4s}}}{(4\pi s)^{\frac{N}{2}}} u^k(z,t) dz ds \\
&\leq \frac{1}{\lambda_2} \left[ (\chi_2\lambda_2\mu_2 - \chi_1\lambda_1\mu_1)_+ + \chi_1\mu_1(\lambda_1 - \lambda_2)_+ \right] (\bar{u} + \varepsilon)^k \\
&\quad - \frac{1}{\lambda_2} \left[ (\chi_2\lambda_2\mu_2 - \chi_1\lambda_1\mu_1)_+ + \chi_1\mu_1(\lambda_1 - \lambda_2)_+ \right] (\underline{u} - \varepsilon)^k.
\end{align*}
Thus, by applying the comparison principle for fractional parabolic equations, we can obtain that
\begin{equation}\label{4.4}
\begin{aligned}
u(x,t) \leq \hat{U}(t)
\end{aligned}
\end{equation}
for all $x \in \mathbb{R}^N$ and $t \geq \tilde{T}_\varepsilon$,
where \( \hat{U} \) be the solution of the following ODE
\begin{equation*}
\left\{
\begin{aligned}
&\hat{U}_t = \bigg( a
+ \frac{1}{\lambda_2} \Bigl[ (\chi_2\lambda_2\mu_2 - \chi_1\lambda_1\mu_1)_+
+ \chi_1\mu_1(\lambda_1 - \lambda_2)_+ \Bigr] (\bar{u} + \varepsilon)^k \\
&\qquad - \frac{1}{\lambda_2} \Bigl[ (\chi_2\lambda_2\mu_2 - \chi_1\lambda_1\mu_1)_+
+ \chi_1\mu_1(\lambda_1 - \lambda_2)_+ \Bigr] (\underline{u} - \varepsilon)^k \\
&\qquad - (b + \chi_2\mu_2 - \chi_1\mu_1) \hat{U}^{\gamma - 1} \bigg) \hat{U}, \\[6pt]
&\hat{U}(\tilde{T}_\varepsilon) = \bigl\| u(\cdot, \tilde{T}_\varepsilon) \bigr\|_{L^\infty}.
\end{aligned}
\right.
\end{equation*}
Since \( b + \chi_2\mu_2 - \chi_1\mu_1 > 0 \), \( \| u(\cdot, \hat{T}_\varepsilon) \|_{L^\infty} > 0 \), it yields that
\begin{equation}\label{4.5}
\begin{aligned}
\lim_{t \to \infty} \hat{U} &\leq \frac{1}{\left( b + \chi_2\mu_2 - \chi_1\mu_1 \right)^{\frac{1}{\gamma - 1}}} \left\{ a + \frac{1}{\lambda_2} \left[ (\chi_2\lambda_2\mu_2 - \chi_1\lambda_1\mu_1)_+ + \chi_1\mu_1(\lambda_1 - \lambda_2)_+ \right] (\bar{u} + \varepsilon)^k \right. \\
&\quad \left. - \frac{1}{\lambda_2} \left[ (\chi_2\lambda_2\mu_2 - \chi_1\lambda_1\mu_1)_+ + \chi_1\mu_1(\lambda_1 - \lambda_2)_+ \right] (\underline{u} - \varepsilon)^k \right\}^{\frac{1}{\gamma - 1}}.
\end{aligned}
\end{equation}
Combining \eqref{4.4} with \eqref{4.5} and letting \( \varepsilon \to 0 \) yields that
\begin{equation}\label{4.7}
\begin{aligned}
\bar{u} &\leq \frac{1}{\left( b + \chi_2\mu_2 - \chi_1\mu_1 \right)^{\frac{1}{\gamma - 1}}} \left\{ a + \frac{1}{\lambda_2} \left[ (\chi_2\lambda_2\mu_2 - \chi_1\lambda_1\mu_1)_+ + \chi_1\mu_1(\lambda_1 - \lambda_2)_+ \right] \bar{u}^k \right. \\
&\quad \left. - \frac{1}{\lambda_2} \left[ (\chi_2\lambda_2\mu_2 - \chi_1\lambda_1\mu_1)_+ + \chi_1\mu_1(\lambda_1 - \lambda_2)_+ \right] \underline{u}^k \right\}^{\frac{1}{\gamma - 1}}.
\end{aligned}
\end{equation}
If
\begin{equation*}
\begin{aligned}
\big\{&a + \frac{1}{\lambda_2} \left[ (\chi_2\lambda_2\mu_2 - \chi_1\lambda_1\mu_1)_+ + \chi_1\mu_1(\lambda_1 - \lambda_2)_+ \right] \bar{u}^k \\
&- \frac{1}{\lambda_2} \left[ (\chi_2\lambda_2\mu_2 - \chi_1\lambda_1\mu_1)_+ + \chi_1\mu_1(\lambda_1 - \lambda_2)_+ \right] \underline{u}^k\big\}^{\frac{1}{\gamma - 1}}=0,
\end{aligned}
\end{equation*}
we can know \( \underline{u} = \bar{u} = 0 \). It follows that
\begin{align*}
0 &= \left\{ a + \frac{1}{\lambda_2} \left[ (\chi_2\lambda_2\mu_2 - \chi_1\lambda_1\mu_1)_+ + \chi_1\mu_1(\lambda_1 - \lambda_2)_+ \right] \bar{u}^k \right. \\
&\quad \left. - \frac{1}{\lambda_2} \left[ (\chi_2\lambda_2\mu_2 - \chi_1\lambda_1\mu_1)_+ + \chi_1\mu_1(\lambda_1 - \lambda_2)_+ \right] \underline{u}^k \right\}^{\frac{1}{\gamma - 1}} = a > 0,
\end{align*}
which is impossible. Therefore, we have
\[
\begin{aligned}
\bar{u} &\leq \frac{1}{\left( b + \chi_2\mu_2 - \chi_1\mu_1 \right)^{\frac{1}{\gamma - 1}}} \left\{ a + \frac{1}{\lambda_2} \left[ (\chi_2\lambda_2\mu_2 - \chi_1\lambda_1\mu_1)_+ + \chi_1\mu_1(\lambda_1 - \lambda_2)_+ \right] \bar{u}^k \right. \\
&\quad \left. - \frac{1}{\lambda_2} \left[ (\chi_2\lambda_2\mu_2 - \chi_1\lambda_1\mu_1)_+ + \chi_1\mu_1(\lambda_1 - \lambda_2)_+ \right] \underline{u}^k \right\}^{\frac{1}{\gamma - 1}}.
\end{aligned}
\]
Moreover, for every \( x \in \mathbb{R}^N, t \geq \tilde{T}_\varepsilon \), it can be inferred that
\begin{equation*}
\begin{aligned}
(\chi_2\lambda_2 w - \chi_1\lambda_1 v)(x,t) &\geq \frac{1}{\lambda_2} \left[ (\chi_2\lambda_2\mu_2 - \chi_1\lambda_1\mu_1)_+ + \chi_1\mu_1(\lambda_1 - \lambda_2)_+ \right] (\underline{u} - \varepsilon)^k \\
&\quad - \frac{1}{\lambda_2} \left[ (\chi_2\lambda_2\mu_2 - \chi_1\lambda_1\mu_1)_+ + \chi_1\mu_1(\lambda_1 - \lambda_2)_+ \right] (\bar{u} + \varepsilon)^k.
\end{aligned}
\end{equation*}
Therefore, we have that
\[
\begin{aligned}
u_t &\geq -(-\Delta)^\alpha u + \nabla (\chi_2\lambda_2 w - \chi_1\lambda_1 v) \nabla u\\
&\quad + \left( a + \frac{1}{\lambda_2} \left[ (\chi_2\lambda_2\mu_2 - \chi_1\lambda_1\mu_1)_+ + \chi_1\mu_1(\lambda_1 - \lambda_2)_+ \right] (\underline{u} - \varepsilon)^k \right. \\
&\quad \left. - \frac{1}{\lambda_2} \left[ (\chi_2\lambda_2\mu_2 - \chi_1\lambda_1\mu_1)_+ + \chi_1\mu_1(\lambda_1 - \lambda_2)_+ \right] (\bar{u} + \varepsilon)^k - (b + \chi_2\mu_2 - \chi_1\mu_1) u^{\gamma - 1} \right) u.
\end{aligned}
\]
Let \( \underline{U} \) be the solution of the following ODE
\begin{equation}\label{4.9}\allowdisplaybreaks
\begin{cases}
\underline{U}_t= \Big( a + \frac{1}{\lambda_2} \left[ (\chi_2\lambda_2\mu_2 - \chi_1\lambda_1\mu_1)_+ + \chi_1\mu_1(\lambda_1 - \lambda_2)_+ \right] (\underline{u} - \varepsilon)^k  \\
\qquad~~  - \frac{1}{\lambda_2} \left[ (\chi_2\lambda_2\mu_2 - \chi_1\lambda_1\mu_1)_+ + \chi_1\mu_1(\lambda_1 - \lambda_2)_+ \right] (\bar{u} + \varepsilon)^k \\
\qquad~~- (b + \chi_2\mu_2 - \chi_1\mu_1) \underline{U}^{\gamma - 1} \Big) \underline{U}, \\
\underline{U}(\tilde{T}_\varepsilon)= \inf_{x \in \mathbb{R}^N} u(\cdot, \tilde{T}_\varepsilon).
\end{cases}
\end{equation}
Thus, by \cref{lemma:Lemma 4.1}, we get \( \inf_{x \in \mathbb{R}^N} u(x,t) > 0 \). Furthermore, by applying the comparison principle for fractional parabolic equations, we conclude that
\begin{equation}\label{4.10}
u(x,t) \geq \underline{U}(t)
\end{equation}
for $x \in \mathbb{R}^N$ and $t \geq \tilde{T}_\varepsilon$. Since \( b + \chi_2\mu_2 - \chi_1\mu_1 > 0 \) and \( \| u(\cdot, \hat{T}_\varepsilon) \|_{L^\infty} > 0 \) for every \( t \geq \tilde{T}_\varepsilon \), combining with \eqref{4.9} and \eqref{4.10} and letting \( \varepsilon \to 0 \), we can deduce that
\begin{equation}\label{4.13}
\begin{aligned}
\underline{u} &\geq \frac{1}{\left( b + \chi_2\mu_2 - \chi_1\mu_1 \right)^{\frac{1}{\gamma - 1}}} \left\{ a + \frac{1}{\lambda_2} \left[ (\chi_2\lambda_2\mu_2 - \chi_1\lambda_1\mu_1)_+ + \chi_1\mu_1(\lambda_1 - \lambda_2)_+ \right] \underline{u}^k \right. \\
&\quad \left. - \frac{1}{\lambda_2} \left[ (\chi_2\lambda_2\mu_2 - \chi_1\lambda_1\mu_1)_+ + \chi_1\mu_1(\lambda_1 - \lambda_2)_+ \right] \bar{u}^k \right\}^{\frac{1}{\gamma - 1}}.
\end{aligned}
\end{equation}
Using \eqref{4.7} and \eqref{4.13}, we conclude
\begin{equation*}
\left( b + \chi_2\mu_2 - \chi_1\mu_1 - \frac{1}{\lambda_2} \left[ |\chi_2\lambda_2\mu_2 - \chi_1\lambda_1\mu_1| + \chi_1\mu_1|\lambda_1 - \lambda_2| \right] \right) (\bar{u}^{\gamma - 1} - \underline{u}^{\gamma - 1}) \leq 0.
\label{eq:inequality-condition} 
\end{equation*}
It indicates that \( \underline{u} = \bar{u} \). Together with \eqref{4.7} and \eqref{4.13}, it can be concluded that
$ \underline{u} = \bar{u} = \left( \frac{a}{b} \right)^{\frac{1}{\gamma - 1}}$.

(2) Assume \( b + \chi_2\mu_2 - \chi_1\mu_1 - \frac{1}{\lambda_2} \left[ |\chi_1\mu_1\lambda_1 - \chi_2\mu_2\lambda_2| + \chi_2\mu_2|\lambda_1 - \lambda_2| \right] > 0 \).

By the straight-forward calculation, we can obtain that
\begin{align*}
(\chi_2\lambda_2 w - \chi_1\lambda_1 v)(x,t) &\leq \frac{1}{\lambda_1} \left[ (\chi_2\lambda_2\mu_2 - \chi_1\lambda_1\mu_1)_+ + \chi_2\mu_2(\lambda_1 - \lambda_2)_+ \right] (\bar{u} + \varepsilon)^k \\
&\quad - \frac{1}{\lambda_1} \left[ (\chi_2\lambda_2\mu_2 - \chi_1\lambda_1\mu_1)_+ + \chi_2\mu_2(\lambda_1 - \lambda_2)_+ \right] (\underline{u} - \varepsilon)^k
\end{align*}
and
\begin{align*}
(\chi_2\lambda_2 w - \chi_1\lambda_1 v)(x,t) &\geq \frac{1}{\lambda_1} \left[ (\chi_2\lambda_2\mu_2 - \chi_1\lambda_1\mu_1)_+ + \chi_2\mu_2(\lambda_1 - \lambda_2)_+ \right] (\underline{u} - \varepsilon)^k \\
&\quad - \frac{1}{\lambda_1} \left[ (\chi_2\lambda_2\mu_2 - \chi_1\lambda_1\mu_1)_+ + \chi_2\mu_2(\lambda_1 - \lambda_2)_+ \right] (\bar{u} + \varepsilon)^k.
\end{align*}
Similarly, it is easy to get
\begin{align}\label{4.17}
\bar{u} &\leq \frac{1}{\left( b + \chi_2\mu_2 - \chi_1\mu_1 \right)^{\frac{1}{\gamma - 1}}} \left\{ a + \frac{1}{\lambda_1} \left[ (\chi_2\lambda_2\mu_2 - \chi_1\lambda_1\mu_1)_+ + \chi_2\mu_2(\lambda_1 - \lambda_2)_+ \right] \bar{u}^k \right. \nonumber\\
&\quad \left. - \frac{1}{\lambda_1} \left[ (\chi_2\lambda_2\mu_2 - \chi_1\lambda_1\mu_1)_+ + \chi_2\mu_2(\lambda_1 - \lambda_2)_+ \right] \underline{u}^k \right\}^{\frac{1}{\gamma - 1}}
\end{align}
and
\begin{equation}\label{4.18}
\begin{aligned}
\underline{u} &\geq \frac{1}{\left( b + \chi_2\mu_2 - \chi_1\mu_1 \right)^{\frac{1}{\gamma - 1}}} \left\{ a + \frac{1}{\lambda_1} \left[ (\chi_2\lambda_2\mu_2 - \chi_1\lambda_1\mu_1)_+ + \chi_2\mu_2(\lambda_1 - \lambda_2)_+ \right] \underline{u}^k \right. \\
&\quad \left. - \frac{1}{\lambda_1} \left[ (\chi_2\lambda_2\mu_2 - \chi_1\lambda_1\mu_1)_+ + \chi_2\mu_2(\lambda_1 - \lambda_2)_+ \right] \bar{u}^k \right\}^{\frac{1}{\gamma - 1}}.
\end{aligned}
\end{equation}
Combining estimates \eqref{4.17} and \eqref{4.18} yields
\[
\left( b + \chi_2\mu_2 - \chi_1\mu_1 - \frac{1}{\lambda_1} \left[ |\chi_2\lambda_2\mu_2 - \chi_1\lambda_1\mu_1| + \chi_2\mu_2|\lambda_1 - \lambda_2| \right] \right) (\bar{u}^{\gamma - 1} - \underline{u}^{\gamma - 1}) \leq 0.
\]
It implies that \( \underline{u} = \bar{u} \). This along with \eqref{4.17} and \eqref{4.18}, it can be concluded that
$\underline{u} = \bar{u} = \left( \frac{a}{b} \right)^{\frac{1}{\gamma - 1}}$.

\noindent When \( \gamma \neq k + 1 \).

Suppose that the parameters $\chi_i$, $\mu_i$ and $\lambda_i$ ($i=1,2$) satisfy
\[
\chi_2\mu_2 - \chi_1\mu_1 = 0, \, \lambda_1 - \lambda_2 = 0.
\]
Then we can deduce that $\chi_2\lambda_2 w - \chi_1\lambda_1 v = 0$.
Similar to proof of Case 1, we conclude that \( \underline{u} = \bar{u} = \left( \frac{a}{b} \right)^{\frac{1}{\gamma - 1}} \). This completes the proof.
$\hfill\Box$

\section{Spreading Speed}
We devote this section to exploring the propagation speed of the global classical solution to \eqref{1.1}.
Suppose that \( u_0 \in C_{unif}^b(\mathbb{R}^N) \), \( \inf\limits_{x \in \mathbb{R}^N} u_0(x) \geq 0 \), \( \alpha \in \left( \frac{1}{2}, 1 \right) \), \( N > 2\alpha \) and \( \chi_i, \mu_i, \lambda_i \) ($i=1,2$) are as defined in \cref{thm:mytheorem1.4}.

Let \( D_l = \left\{ x \in \mathbb{R}^N \mid |x_i| < l \text{ for } i = 1, 2, \ldots, N \right\} \)
and \( B^{N-1} = \left\{ x \in \mathbb{R}^N \mid |x| = 1 \right\} \). For any unit vector \( \xi \in B^{N-1}\) and constant \(c \in \mathbb{R} \), we introduce the shifted functions
\begin{equation*}
\left\{
\begin{aligned}
\hat{u}(x,t) = u(x + e^{c \xi t}, t),\\
\hat{v}(x,t) = v(x + e^{c \xi t}, t),\\
\hat{w}(x,t) = w(x + e^{c \xi t}, t)
\end{aligned}
\right.
\end{equation*}
and consider the following system
\begin{equation}\label{5.1}
\left\{
\begin{aligned}
\hat{u}_t &= -(-\Delta)^\alpha \hat{u} + c \xi e^{c \xi t} \cdot \nabla \hat{u}
- \chi_1 \nabla \cdot (\hat{u} \nabla \hat{v}) + \chi_2 \nabla \cdot (\hat{u} \nabla \hat{w})
+ a \hat{u} - b \hat{u}^\gamma, &x \in \mathbb{R}^N,\\
0 &= \Delta \hat{v}- \lambda_1\hat{v} + \mu_1 \hat{u}^k, &x \in \mathbb{R}^N,\\
0 &= \Delta \hat{w}- \lambda_2\hat{w} + \mu_2 \hat{u}^k, &x \in \mathbb{R}^N.
\end{aligned}
\right.
\end{equation}
Let \( (\hat{u}, \hat{v}, \hat{w}) \) be a classical solution of \eqref{5.1}
with $ \hat{u}(x,0) = u_0(x) $ and $ \inf_{x \in \mathbb{R}^N} u_0(x) \geq 0 $.
\begin{lemma}\label{le4.1.0}
For any \( \xi \in B^{N-1}\) and \(c \in \mathbb{R} \), there exist \( C_2 > 0\), \(\hat{T_0} \gg 1 \) and \( 0 < \omega < 1 \)
such that
\[
\begin{cases}
\|\hat{u}(\cdot, t)\|_{L^\infty} \leq C_0,\\
\|\hat{v}(\cdot, t)\|_{L^\infty}\leq \frac{\mu_1}{\lambda_1} C_0^k,~~ \|\hat{w}(\cdot, t)\|_{L^\infty} \leq \frac{\mu_2}{\lambda_2} C_0^k, \\
\|\nabla \hat{v}(\cdot, t)\|_{L^\infty}\leq \frac{\mu_1 \sqrt{N}}{\sqrt{\lambda_1}} C_0^k,~~ \|\nabla \hat{w}(\cdot, t)\|_{L^\infty} \leq \frac{\mu_2 \sqrt{N}}{\sqrt{\lambda_2}} C_0^k
\end{cases}
\]
for all $t\geq\hat{T_0}$ and
\begin{equation}\label{5.4}
\begin{aligned}
\Bigg\{\sup_{\substack{t, s \geq \hat{T}_0+1 \\ t \neq s}} \frac{\|\nabla \hat{v}(\cdot, t) - \nabla \hat{v}(\cdot, s)\|_{L^\infty}}{|t - s|^\omega}, \sup_{\substack{t, s \geq \hat{T}_0+1 \\ t \neq s}} \frac{\|\nabla \hat{w}(\cdot, t) - \nabla \hat{w}(\cdot, s)\|_{L^\infty}}{|t - s|^\omega}\Bigg\} \leq {C}_2,
\end{aligned}
\end{equation}
where the definition of \( C_0 \) is the same as that in \eqref{1.1} of \cref{thm:mytheorem1.2}.
\end{lemma}

\begin{proof}
By direct computations, we get
\begin{align*}
\|\hat{v}(\cdot, t)\|_{L^\infty} &= \mu_1 \int_0^\infty \int_{\mathbb{R}^N} \frac{e^{-\frac{|y - e^{c(t-s)\xi} - x|^2}{4(t - s)}}}{(4\pi(t - s))^{\frac{N}{2}}} \hat{u}^k(y, s) \, dy\, ds \\
&\leq \mu_1 \pi^{-\frac{N}{2}} \|\hat{u}^k(\cdot, t)\|_{L^\infty} \int_0^\infty e^{-\lambda_1 s} \, ds \int_{\mathbb{R}^N} e^{-|z|^2} \, dz \\
&\leq \frac{\mu_1}{\lambda_1} C_0^k
\end{align*}
and
\begin{align*}
\|\partial_{y_j} \hat{v}(\cdot, t)\|_{L^\infty} &= \mu_1 \int_0^\infty \int_{\mathbb{R}^N} \frac{(y_j - x_j - e^{c(t-s)\xi}) e^{-\lambda_1 s} e^{-\frac{|x + e^{c(t-s)\xi} - y|^2}{4(t - s)}}}{2(t - s)(4\pi(t - s))^{\frac{N}{2}}} \hat{u}^k(y, s) \, dy\, ds \\
&\leq \mu_1 \pi^{-\frac{N}{2}} \|\hat{u}^k(\cdot, t)\|_{L^\infty} \int_0^\infty e^{-\lambda_1 s} s^{-\frac{1}{2}} \, ds \int_{\mathbb{R}^N} |z| e^{-|z|^2} \, dz \\
&= \frac{\mu_1}{\sqrt{\lambda_1}} C_0^k,
\end{align*}
where \( \int_{\mathbb{R}^N} e^{-|x|^2} dx = \pi^{\frac{N}{2}} \) and
\( \int_{\mathbb{R}^N} x_i e^{-|x|^2} dx = \pi^{\frac{N-1}{2}} \).
Since \( |\nabla \hat{v}| = \sqrt{\sum_{j=1}^N |\partial_{x_j} \hat{v}|^2} \), it follows that
\[
\|\nabla \hat{v}\|_{L^\infty} \leq \frac{\mu_1 \sqrt{N}}{\sqrt{\lambda_1}} C_0^k.
\]
The estimate for $w$ follows similarly. Furthermore, \eqref{5.4} follows by arguments similar to those for Theorem 1.1(iii) in \cite{JLL}.
\end{proof}

\noindent \textbf{The proof of Theorem 1.3.} The proof is divided into three steps.

To this end, for fixed \( {\eta} \), we take \( T \geq 1 \) and choose \( L = L({\eta}) \geq l \)
such that \( D_l \subset B_L(0) \) and
\begin{equation}\label{5.3}
\begin{aligned}
\max\left\{
\int_{\mathbb{R}^N \setminus B_{\frac{L - e^{\frac{Ta}{N+2\alpha}}}{2\sqrt{T}}}(0)} e^{-|z|^2} \, dz, \,
\int_{\mathbb{R}^N \setminus B_{\frac{L - e^{\frac{Ta}{N+2\alpha}}}{2\sqrt{T}}}(0)} |z| e^{-|z|^2} \, dz
\right\}
\le {\eta}^k.
\end{aligned}
\end{equation}

\noindent \textbf{Step 1.} First, we give some preliminary estimates in Step 1, which are divided into Claim 1-Claim 3.

\noindent Claim 1. \textit{For any given \( 0 < \varepsilon < \frac{a}{2(N + 2\alpha)} \). There exists \( \varepsilon_0 > 0 \) such that for any \( 0 < \eta < \varepsilon_0 \), \( \xi \in B^{N-1} \), \( -\frac{a}{2(N + 2\alpha)} + \varepsilon \leq c \leq \frac{a}{2(N + 2\alpha)} - \varepsilon \) and any \( t_1, t_2 \) satisfying \( T(u_0) \leq t_1 < t_2 \leq +\infty \),  if
\[\sup_{x\in B_{2L(\eta)}} \hat{u}(x,t) \leq {\eta},\]
then
\[\sup_{x\in B_{2L(\eta)}} \max\left\{\hat{v}(x,t), \hat{w}(x,t), \left|\partial_{x_j} \hat{v}(x,t)\right|, \left|\partial_{x_j} \hat{w}(x,t)\right|\right\} \leq M_1 {\eta}^k \]
for all $t_1 \leq t < t_2$,
where \( B_{2L(\eta)} \) is any ball in $\mathbb{R}^N$ with radius \( 2L(\eta) \) and
\begin{align*}
M_1 := \max\left\{ \frac{\mu_i}{\lambda_i} + \frac{\mu_i C_0^k}{\lambda_i \pi^{\frac{N}{2}}}, \frac{\mu_i}{\sqrt{\lambda_i}} + \frac{\mu_i C_0^k}{\sqrt{\lambda_i}} \pi^{\frac{1-N}{2}} \right\} ~~\text{ for } i = 1,2 \text{ and } j = 1,2,\ldots,N.
\end{align*}}

Fix $t_1 \geq T(u_0)$. By a simple calculation, we can get
\begin{align}\label{5.3.1}
\|\hat{v}(x,t)\|_{L^\infty} &= \mu_1 \int_{t_1}^t \int_{\mathbb{R}^N} \frac{e^{-\lambda_1 (t-s)}}{(4\pi (t-s))^{\frac{N}{2}}} e^{-\frac{|x + e^{c(t-s)\xi} - y|^2}{4(t-s)}} \hat{u}^k(y,s) \, dy \, ds \nonumber\\
&= \frac{\mu_1}{\pi^{\frac{N}{2}}} \int_{t_1}^t e^{-\lambda_1 (t-s)} ds \int_{\mathbb{R}^N} e^{-|z|^2} \hat{u}^k ( x + e^{c(t-s)\xi} + 2\sqrt{t-s} \, z, s) dz \nonumber\\
&\leq \frac{\mu_1}{\pi^{\frac{N}{2}}} \|\hat{u}^k(\cdot,t)\|_{L^\infty} \int_{t_1}^t e^{-\lambda_1 (t-s)} ds \int_{\mathbb{R}^N} e^{-|z|^2} dz \nonumber\\
&\quad + \mu_1 C_0^k \int_{t_1}^t e^{-\lambda_1 (t-s)} ds \int_{\mathbb{R}^N \setminus B_{\frac{L - e^{\frac{Ta}{N+2\alpha}}}{2\sqrt{T}}}} e^{-|z|^2} dz \nonumber\\
&\leq \left( \frac{\mu_1}{\lambda_1} + \frac{\mu_1 C_0^k}{\lambda_1 \pi^{\frac{N}{2}}} \right) {\eta}^k
\end{align}
and
\begin{align}\label{5.3.2}
\|\partial_{x_j} \hat{v}(x,t)\|_{L^\infty} &= \mu_1 \int_{t_1}^t \int_{\mathbb{R}^N} \frac{(x_j + e^{c(t-s)\xi}_j - y_j) e^{-\lambda_1 (t-s)} e^{-\frac{|x + e^{c(t-s)\xi} - y|^2}{4(t-s)}}}{2(t-s) (4\pi (t-s))^{\frac{N}{2}}} \hat{u}^k(y,s) \, dy \, ds \nonumber\\
&= \frac{\mu_1}{\pi^{\frac{N}{2}}} \int_{t_1}^t e^{-\lambda_1 (t-s)} \sqrt{t-s} \, ds \int_{\mathbb{R}^N} z_j e^{-|z|^2} \hat{u}^k ( x + e^{c(t-s)\xi} + 2\sqrt{t-s} \, z, s ) dz \nonumber\\
&\leq \frac{\mu_1}{\pi^{\frac{N}{2}}}\sup_{t_1 \leq t \leq t_2,|x| \leq 2L} \hat{u}^k(x,t) \int_{t_1}^t e^{-\lambda_1 (t-s)} {(t-s)}^{-\frac{1}{2}} \, ds \int_{\mathbb{R}^N} |z| e^{-|z|^2} dz \nonumber\\
&\quad + \mu_1 C_0^k \int_{t_1}^t e^{-\lambda_1 (t-s)} {(t-s)}^{-\frac{1}{2}} \, ds \int_{\mathbb{R}^N \setminus B_{\frac{L - e^{\frac{Ta}{N+2\alpha}}}{2\sqrt{T}}}} |z| e^{-|z|^2} dz \nonumber\\
&\leq \left( \frac{\mu_1}{\sqrt{\lambda_1}} + \frac{\mu_1 C_0^k}{\sqrt{\lambda_1}} \pi^{\frac{1-N}{2}} \right) {\eta}^k
\end{align}
for $i = 1, 2$ and $j = 1, 2, \ldots, N$.
By \eqref{5.3.1} and \eqref{5.3.2}, we can obtain
\[
\sup_{t_1 \leq t \leq t_2 , |x| \leq 2L} \max \left\{ \hat{v}(x,t), \partial_{x_j} \hat{v}(x,t) \right\} \leq \max \left\{ \left( \frac{\mu_1}{\lambda_1} + \frac{\mu_1 C_0^k}{\lambda_1 \pi^{\frac{N}{2}}} \right) {\eta^k}, \left( \frac{\mu_1}{\sqrt{\lambda_1}} + \frac{\mu_1 C_0^k}{\sqrt{\lambda_1}} \pi^{\frac{1 - N}{2}} \right) {\eta^k} \right\}.
\]
The corresponding estimate for $w$ can be derived analogously.

\noindent Claim 2.
\textit{For any given \( 0 < \varepsilon < \frac{a}{2(N + 2\alpha)} \). Let \( \tilde{T}_0 \geq 1 \) be such that \( e^{-\lambda_1^\alpha \tilde{T}_0} \geq 4 \), here \( \lambda_1^\alpha \) does not depend on \( \tilde{T}_0 \). For any \( 0 < \eta < \varepsilon_0 \), there is \( 0 < \delta_\eta < \varepsilon_0 \) such that for any \( \xi \in B^{N-1} \), \( -\frac{a}{2(N + 2\alpha)} + \varepsilon \leq c \leq \frac{a}{2(N + 2\alpha)} - \varepsilon \) and \( t_0 \geq T(u_0) + 2 \), if
\[
\sup_{x \in B_{2L(\eta)}} \hat{u}(x, t) \geq \eta,
\]
then
\[
\inf_{x \in B_{2L(\eta)}} \hat{u}(x, t) \geq \delta_{\eta}
\]
for all $t_0 \leq t < t_0 + T(\eta) + \tilde{T}_0$.}

Since the proof is similar to that of Lemma 6.2 in \cite{JLL}, we omit it here.

To prove Claim 3, we first introduce the eigenvalue problem for the following system
\begin{equation}\label{5.1.1}
\begin{cases}
\begin{aligned}
&(-\Delta)^\alpha u(x) = \lambda^\alpha u(x), && x \in D_l, \\
&u(x) = 0, && x \in \mathbb{R}^N \setminus D_l.
\end{aligned}
\end{cases}
\end{equation}
As $N>2\alpha$, we can derive that the principal eigenvalue \( \lambda_1^\alpha \) of \eqref{5.1.1} is positive and simple,
with the associated principal eigenfunction \( \phi_1(x) \) being positive throughout $D_l$.
Based on these conclusions, we consider the subsequent eigenvalue problem
\begin{equation}\label{5.2}
\begin{cases}
\begin{aligned}
&(-\Delta)^\alpha u(x) - c \xi e^{\xi \tilde{T}_0} \cdot \nabla u(x) - \bar{a} u(x) = \lambda^\alpha u(x), \quad &&x \in D_l, \\
&u(x)= 0, \quad &&x \in \mathbb{R}^N \setminus D_l,
\end{aligned}
\end{cases}
\end{equation}
where \( \tilde{T}_0 \geq 1 \) satisfies \( e^{-\lambda_1^\alpha \tilde{T}_0} \geq 4 \).
Next, we define the fractional Sobolev space \( H_0^\alpha(D_l) \) equipped with the norm
\begin{align}\label{5.2.1}
    \| u \|_{H_0^\alpha(D_l)} = \left( \int_{Q} \frac{|u(x) - u(y)|^2}{|x - y|^{N + 2\alpha}} dxdy \right)^{\frac{1}{2}}
\end{align}
and \( {Q} = \mathbb{R}^{2N} \setminus ((\mathbb{R}^N \setminus D_l) \times (\mathbb{R}^N \setminus D_l)) \). As we can know from \cite{JLL}, \( H_0^\alpha(D_l) \) is a Hilbert space and
\begin{align} \label{5.2.2}
 \langle s, z \rangle_{H_0^\alpha(D_l)} = \int_{Q} \frac{(s(x) - s(y))(z(x) - z(y))}{|x - y|^{N + 2\alpha}} dxdy .
\end{align}
Given that \( s \) and \( z \) belong to \( H_0^\alpha(D_l) \) (implying \( s, z = 0 \) a.e. in \( \mathbb{R}^N \setminus D_l \)), \eqref{5.2.1} and \eqref{5.2.2} can be extended to the entire space \( \mathbb{R}^{2N} \).
Then we focus on the following weak formulation of \eqref{5.2}
\begin{equation}\label{5.2.3}
\begin{cases}
\begin{aligned}
&~~~~\frac{1}{2} \int_{\mathbb{R}^N} \frac{(u(x) - u(y))(\phi(x) - \phi(y))}{|x - y|^{N + 2\alpha}} dxdy - \int_{D_l} c \xi e^{c \xi  \tilde{T}_0} \cdot \nabla u(x) \phi(x) dx - \int_{D_l} \bar{a} u(x) \phi(x) dx \\
&= \lambda^\alpha (c, \bar{a}) \int_{D_l} u(x) \phi(x) dx, \,  \quad\text{for all } \phi(x) \in H_0^\alpha(D_l), \\
&~~~~u(x) \in H_0^\alpha(D_l),
\end{aligned}
\end{cases}
\end{equation}
which combined with Lemma 6.1 of \cite{JLL}, we conclude that the principal eigenvalue \( \lambda_1^\alpha(c, \bar{a}) < 0 \)
of \eqref{5.2.3} when \( \lambda_1^\alpha < \bar{a} < a \).
Moreover, the negativity of the eigenvalue \( \lambda_1^\alpha(c, \bar{a}) < 0 \) ensures that the sub-solution of the system increases over time, which also supports the proof of persistence and the lower bound of the spreading speed.

\noindent Claim 3.
\textit{For any given \( 0 < \varepsilon < \frac{a}{2(N + 2\alpha)} \). Assume \( \lambda_1^\alpha < \bar{a} < a \) and take \( l \) as in \eqref{5.3}. We can find \( 0 < \tilde{\varepsilon}_0 \leq \varepsilon_0 \) satisfying for every \( 0 < \eta < \tilde{\varepsilon}_0 \), \( \xi \in B^{N-1} \), \( -\frac{a}{2(N + 2\alpha)} + \varepsilon \leq c \leq \frac{a}{2(N + 2\alpha)} - \varepsilon \) and any \( t_1, t_2 \) satisfying \( T(u_0) \leq t_1 < t_2 \leq +\infty \), if
\begin{align*}
\sup_{x \in B_{2L(\eta)}} \hat{u}(x, t_1) = \eta, \quad \sup_{x \in B_{2L(\eta)}} \hat{u}(x, t) \leq \eta,
\end{align*}
then
\begin{align*}
\inf_{x \in B_{2L(\eta)}} \hat{u}(x, t) \geq \tilde{\delta}_\eta
\end{align*}
for all $t_1 \leq t < t_2$.}

Define \( \underline{u}(x,t) \) as the solution to the ensuing system
\[
\begin{cases}
\underline{u}_t = -(-\Delta)^\alpha \underline{u} + c \xi e^{c t \xi} \cdot \nabla \underline{u} + q(x,t) \cdot \nabla \underline{u} + \bar{a} \underline{u}, & x \in D_l, t > 0, \\
\underline{u}(x,t) = 0, & x \in \partial D_l, t > 0, \\
\underline{u}(x,0) = \phi_1(x; \xi, c, \bar{a}), & x \in D_l,
\end{cases}
\]
where \( \phi_1(x; \xi, c, \bar{a}) = \frac{\phi_1}{\|\phi_1\|_{L^\infty}} \), where \( \phi_1 \) is the principal eigenfunction.

We assert that there exists some \( \tilde{\varepsilon}_0 > 0 \) such that, for every function \( q(x,t) \) that is \( C^1 \) in the spatial variable \( x \) and H\"{o}lder continuous in time \( t \) with exponent \( 0 < \theta < 1 \),
\begin{equation}\label{5.5}
\begin{aligned}
\sup_{t > 0} \|q(\cdot, t)\|_{C(D_l)} \leq \left( \chi_1+ \chi_2 \right) \sqrt{N}{M_2}\tilde{\varepsilon}_0^k
\end{aligned}
\end{equation}
and
\begin{equation}\label{5.6}
\begin{aligned}
\sup_{t > 0} \|q(\cdot, t)\|_{C^\theta} \leq (\chi_1 + \chi_2) C_2,
\end{aligned}
\end{equation}
then
\[
\underline{u}(x, \tilde{T}_0) \geq 2\phi_1(x).
\]
Following the argument  in Lemma 6.5 of \cite{JLL}, we prove this claim by contradiction and omit the details here.

Without loss of generality, we may assume that
\[
a - \chi_1\lambda_1 M_2 \tilde{\varepsilon}_0^k - b \tilde{\varepsilon}_0^{\gamma - 1} + (\chi_1\mu_1 - \chi_2\mu_2) \tilde{\varepsilon}_0^k \geq \bar{a}.
\]
Let \( T = T(\eta) \), it follows from Claim 1 that
\[
\begin{aligned}
\hat{u}_t &= -(-\Delta)^\alpha \hat{u} + c \xi e^{c t \xi} \cdot \nabla \hat{u} + (\nabla \chi_2 \hat{w} - \nabla \chi_1 \hat{v}) \cdot \nabla \hat{u} \\
&~\quad + \hat{u} \left( a - b \hat{u}^{\gamma - 1} + (\chi_2\mu_2 - \chi_1\mu_1) \hat{u}^k \right) \\
&\geq -(-\Delta)^\alpha \hat{u} + c \xi e^{c t \xi} \cdot \nabla \hat{u} + (\nabla \chi_2 \hat{w} - \nabla \chi_1 \hat{v}) \cdot \nabla \hat{u} + \bar{a} \hat{u}
\end{aligned}
\]\\
for any \( 0 < \eta < \tilde{\varepsilon}_0 \), \( \xi \in B^{N-1} \), \( -\frac{a}{2(N + 2\alpha)} + \varepsilon \leq c \leq \frac{a}{2(N + 2\alpha)} - \varepsilon \), \( t_1 + T + 1 \leq t < t_2 \leq +\infty \) and \( x \in B_L(0) \), where $q(x, t) = \nabla \chi_2 \hat{w} - \nabla \chi_1 \hat{v}$. By Claim 1 and Claim 2, we can deduce that $q(\cdot, t_1 + T +1)$ satisfies \eqref{5.5} and \eqref{5.6}.

Let \( n_0 \geq 0 \) be such that $t_1 + T + 1 + n_0 \tilde{T}_0 < t_2$ and $t_1 + T + 1 + (n_0 + 1)\tilde{T}_0 \geq t_2$, it follows from Claim 2 that
\[
\inf_{\substack{x \in B_{2L(0)} \\ }} \hat{u}(x,t) \geq \delta_{\eta}
\]
for all $t_1 < t < t_1 + T + 1$.
Combining the above claim with the comparison principle, we can derive
\[
\hat{u}(x, t_1 + T + 1 + m \hat{T}_0; \xi, c, u_0) \geq 2^m \delta_{\eta} \phi_1(x; \xi, c, \bar{a})~~ \text{ for }  m = 1, 2, \ldots, n_0
\]
for all $x \in D_l$.
Similar to the proof in Claim 2, we have $\inf_{x \in B_{2L}(0)} \hat{u}(x, t) \geq \delta_{\delta_{\eta}}$ for all $t_1 < t < t_1 + T + 1 + (m + 1)\hat{T}_0$.
Thus, for any \(\xi \in B^{N - 1}\) and \(-\frac{a}{2(N + 2\alpha)} + \varepsilon \leq c \leq \frac{a}{2(N + 2\alpha)} - \varepsilon\), we can deduce that
\[
\inf_{x \in B_{2L}(0)} \hat{u}(x, t) \geq \tilde{\delta}_{\eta} := \min\{\delta_{\eta}, \delta_{\delta_{\eta}}\}
\]
for all $t_1 < t < t_2$.

\noindent \textbf{Step 2.} Next, we demonstrate the lower bound of the spreading speed.

Assume that \(\tilde{\delta} := \tilde{\delta}(\xi, c) = \inf_{x \in \bar{D}_l} \hat{u}(T_0 + T(\tilde{\varepsilon}_0) + 1, \xi, c, u_0)\), \(u_0(x) > 0\), \(\tilde{\delta} > 0\), \(m_0 = \inf \{ m \in \mathbb{Z}^+ \mid 2^m \tilde{\delta} \geq \tilde{\varepsilon}_0 \}\) and \(T_{00} = T_0 + T(\tilde{\varepsilon}_0) + 1 + m_0 \tilde{T}_0\).

Next, we claim that for any \( \xi \in B^{N-1} \) and \( -\frac{a}{2(N + 2\alpha)} + \varepsilon \leq c \leq \frac{a}{2(N + 2\alpha)} - \varepsilon \), we have
\begin{equation}\label{5.7}
\begin{aligned}
\inf_{\substack{|x| \leq 2L(\tilde{\varepsilon}_0) \\}} \hat{u}(x,t) \geq \min\{\delta_{\tilde{\varepsilon}_0}, \tilde{\delta}_{\tilde{\varepsilon}_0}\}
\end{aligned}
\end{equation}
for all $t \geq T_{00}$.

To this end, we first introduce an open set \( E = \left\{ t > T_0 \,\big|\, \sup_{\substack{|x| \leq 2L(\tilde{\varepsilon}_0)}} \hat{u}(x, t) \leq \tilde{\varepsilon}_0 \right\} \), it follows from Claim 2 in Step 1 that
\begin{equation}\label{5.8}
\begin{aligned}
\inf_{\substack{|x| \leq 2L(\tilde{\varepsilon}_0)}} \hat{u}(x, t) \geq \delta_{\tilde{\varepsilon}_0}
\end{aligned}
\end{equation}
for all $t \notin E$. We shall discuss the claim by dividing it into two cases.

\noindent Case 1. \( E \) is an empty set.

From Claim 2 in Step 1, we have
\begin{equation}\label{5.9}
\begin{aligned}
\inf_{\substack{|x| \leq 2L({\tilde{\varepsilon}_0}) \\}} \hat{u}(x,t) \geq \delta_{\tilde{\varepsilon}_0}
\end{aligned}
\end{equation}
for all $t \geq T_0$.

\noindent Case 2. \( E \) is a non-empty set.

Suppose that \( E = \bigcup_{i} (E_i, E_{i + 1}) \), if \( E_i \neq T_0 \), then
\[
\sup_{\substack{|x| \leq 2L(\eta) \\}} \hat{u}(x,E_{i}) = \tilde{\varepsilon}_0, \,~~ \sup_{\substack{|x| \leq 2L(\eta) \\}} \hat{u}(x,t) \leq \tilde{\varepsilon}_0
\]
for all $t \in (E_i, E_{i + 1})$.
It follows from Claim 3 in Step 1 that
\begin{equation}\label{5.10}
\begin{aligned}
\inf_{\substack{|x| \leq 2L(\delta_{\tilde{\varepsilon}_0}) \\}} \hat{u}(x,t) \geq \tilde{\delta}_{\tilde{\varepsilon}_0} ~~\text{ for } E_i \neq T_0
\end{aligned}
\end{equation}
for all $t \in (E_i, E_{i + 1})$.
If \( E_i = T_0 \), by Claim 3 in Step 1, we get
\[
\hat{u}(x, T_0 + T(\tilde{\varepsilon}_0) + 1 + m \tilde{T}_0; \xi, c, u_0) \geq 2^m \delta_{\eta} \phi_1(x; \xi, c, \bar{a}) ~~~\text{for }  m = 0, 1, 2, \ldots, m_0
\]
for all $x \in D_l$.
Thus, we can obtain that \( E_{i + 1} \leq T_0 \). The combination of \eqref{5.8}-\eqref{5.10} yields \eqref{5.7}.
Combining this with \( \hat{u}(x,t) = u(x + e^{c t \xi}, t) \), for any \( t \geq T_0 \), \( \xi \in B^{N - 1} \) and \( -\frac{a}{2(N + 2\alpha)} + \varepsilon \leq c \leq \frac{a}{2(N + 2\alpha)} - \varepsilon \), we can deduce that
\begin{align*}
\inf_{\substack{|x - e^{c t \xi}| \leq 2L(\tilde{\varepsilon}_0) \\}} u(x,t) \geq \min\left\{ \delta_{\tilde{\varepsilon}_0}, \tilde{\delta}_{\tilde{\varepsilon}_0} \right\}.
\end{align*}
Therefore, for any \( t \geq T_0 \) and \( |x| \leq e^{\left( \frac{a}{N + 2\alpha} - \varepsilon \right) t} \), if \( x > 0 \), then there exist \( c = \frac{\ln |x|}{t} \), and \( \xi = \frac{\ln x}{\ln |x|} \) such that \( |x - e^{c t \xi}| \leq 2L(\tilde{\varepsilon}_0) \). If \( x < 0 \), then there exist \( c = \frac{\ln |x|}{t} \), and \( \xi = \frac{\ln (-x)}{\ln |x|} \) such that \( |x - e^{c t \xi}| \leq 2L(\tilde{\varepsilon}_0) \). And if \( x = 0 \), the conclusion holds obviously. Then, for any \( t \geq T_{00} \), we obtain that
\[
\inf_{\substack{|x| \leq e^{\left( \frac{a}{N + 2\alpha} - \varepsilon \right) t} \\}} u(x,t) \geq \min\left\{ \delta_{\tilde{\varepsilon}_0}, \tilde{\delta}_{\tilde{\varepsilon}_0} \right\},
\]
it immediately yields
\[
\liminf_{t \to \infty} \inf_{\substack{|x| \leq e^{\left( \frac{a}{N + 2\alpha} - \varepsilon \right) t} \\}} u(x,t) \geq \min\left\{ \delta_{\tilde{\varepsilon}_0}, \tilde{\delta}_{\tilde{\varepsilon}_0} \right\}.
\]
\textbf{Step 3.} In this step, we prove the upper bound of the spreading speed.

According to the \eqref{1.1}, we have
\begin{align}\label{5.11}
u_t &\leq -(-\Delta)^\alpha u + \nabla (\chi_2 w - \chi_1 v) \cdot \nabla u + u \left( a + M C_0^k - b u^{\gamma - 1} + (\chi_1 \mu_1 - \chi_2 \mu_2) u^k \right).
\end{align}
{Case 1.} \( \gamma = k + 1 \).

When \( \gamma = k + 1 \), \eqref{5.11} takes the form
\[
u_t \leq -(-\Delta)^\alpha u + \nabla (\chi_2 w - \chi_1 v) \cdot \nabla u + u \left( a + M C_0^k - \left( b + \chi_2\mu_2 - \chi_1\mu_1 \right) u^k \right).
\]
Since \( b + \chi_2\mu_2 - \chi_1\mu_1 > M \), we have
\[
u_t \leq -(-\Delta)^\alpha u + \nabla (\chi_2 w - \chi_1 v) \cdot \nabla u + u \left( a + M C_0^k \right).
\]
Let \( \bar{U}(x,t) \) be the solution of the following equation
\[
\begin{cases}
\bar{U}_t + (-\Delta)^\alpha \bar{U} + \nabla (\chi_1 v - \chi_2 w) \cdot \nabla \bar{U} = (a + M C_0^k) \bar{U}, & x \in \mathbb{R}^N, t > 0, \\
\bar{U}(x, 0) = u_0, & x \in \mathbb{R}^N.
\end{cases}
\]
By applying the comparison principle, we get $0 \leq u(x,t) \leq \bar{U}(x,t)$
for all $x \in \mathbb{R}^N$ and $t > 0$.

Let \( \bar{W} = e^{-(a + M C_0^k) t} \bar{U}, {V} = e^{-(a + M C_0^k) t} v, {W} = e^{-(a + M C_0^k) t} w\), we obtain that
\[
\begin{cases}
\bar{W}_t + (-\Delta)^\alpha \bar{W} + e^{(a + M C_0^k) t} \nabla (\chi_1 v - \chi_2 w) \cdot \nabla \bar{W} = 0, & x \in \mathbb{R}^N, t > 0, \\
\bar{W}(x, 0) = u_0, & x \in \mathbb{R}^N.
\end{cases}
\]
From \cref{def: Dedinition 1} and Lemma \ref{le4.1.0}, we derive that
\begin{align*}
K_{|e^{(a + M C_0^k) t} \nabla (\chi_1 v - \chi_2 w)|}^{2\alpha}(r) &= \sup_{x \in \mathbb{R}^N} \int_{B_r(x)} \frac{|\nabla (\chi_1 v - \chi_2 w)|}{|x - y|^{N + 1 - 2\alpha}} dy \\
&\leq \left( \frac{\chi_1 \mu_1}{\sqrt{\lambda_1}} + \frac{\chi_2 \mu_2}{\sqrt{\lambda_2}} \right) \sqrt{N} C_0^k \sup_{x \in \mathbb{R}^N} \int_{B_r(x)} \frac{1}{\tau^{N + 1 - 2\alpha}} \tau^{N - 1} d\tau \\
&\leq \left( \frac{\chi_1 \mu_1}{\sqrt{\lambda_1}} + \frac{\chi_2 \mu_2}{\sqrt{\lambda_2}} \right) \sqrt{N} C_0^k \frac{1}{2\alpha - 1} r^{2\alpha - 1}.
\end{align*}
Since \( \alpha \in \left( \frac{1}{2}, 1 \right) \), then we have \( \lim_{r \to 0} K_{|e^{(a + M C_0^k) t} \nabla (\chi_1 v - \chi_2 w)|}^{2\alpha}(r) = 0 \).\\
By Theorem 1.1 in \cite{JLL}, for any \( x, y \in \mathbb{R}^N \) and \( 0 < t \leq T \), we can obtain that
\begin{align*}
\bar{U}(x,t) = e^{(a + M C_0^k) t} \int_{\mathbb{R}^N} P_b^\alpha(x, y, t) u_0(y) dy
\end{align*}
and
\begin{align*}
\frac{t}{C_3 \left( t^{\frac{1}{2\alpha}} + |x - y| \right)^{N + 2\alpha}} \leq P_b^\alpha(x, y, t) \leq \frac{C_3 t}{\left( t^{\frac{1}{2\alpha}} + |x - y| \right)^{N + 2\alpha}},
\end{align*}
with \( P_b^\alpha(x, y, t) \) representing the fundamental solution corresponding to the operator \( -(-\Delta)^\alpha - e^{(a + M C_0^k)t}\nabla(\chi_1 V - \chi_2 W) \cdot \nabla \).
Let \( r_0 \geq 1 \), By the similar reason in \cite{JLL}, we conclude that
\begin{align*}
u(x,t) \leq \bar{U}(x,t) \leq C^* \left( \frac{r_0^{-2\alpha} t}{N} + 1 \right) e^{(a+M C_0^k) t} |x|^{-N + 2\alpha}.
\end{align*}
Then, by selecting a sufficiently large $t$ and \( |x| \geq e^{\left( \frac{a+M C_0^k}{N + 2\alpha} + \varepsilon \right) t} \), it follows that
\begin{align*}
u(x,t) \leq C^* \left( \frac{r_0^{-2\alpha} t}{N} + 1 \right) e^{-\varepsilon t} \to 0 \quad \text{as }  t \to \infty .
\end{align*}
{Case 2.} \( \gamma \neq k + 1 \).

When \( \gamma \neq k + 1 \), by applying similar arguments as in Case 1, for $ |x| \geq e^{\left( \frac{a}{N + 2\alpha} + \varepsilon \right) t}$, we obtain
\[
u(x,t) \leq C^* \left( \frac{r_0^{-2\alpha} t}{N} + 1 \right) e^{-\varepsilon t} \to 0 \quad \text{as } t \to \infty.
\]
The proof of Theorem \ref{thm:mytheorem1.4} is thus completed.
$\hfill\Box$

\section*{Conflicts of Interest}
Authors have no conflict of interest to declare.






\end{document}